\documentclass[nonblindrev,opre]{informs3}

\usepackage{color} 
\usepackage{color-edits}
\addauthor{zc}{cyan}
\addauthor{bs}{purple}
\addauthor{wx}{brown}

\usepackage{fix-cm}

\usepackage{epstopdf}
\usepackage[caption=false]{subfig} 

\usepackage{xfrac} 

\usepackage{natbib}
 \bibpunct[, ]{(}{)}{,}{a}{}{,}%
 \def\bibfont{\small}%
 \usepackage[dvipsnames]{xcolor}
 \usepackage{tikz,bm}
 \usepackage{endnotes,todonotes}

\usetikzlibrary{graphs,graphs.standard}
\usetikzlibrary{arrows}
\usetikzlibrary{arrows.meta}
\usetikzlibrary{graphs,graphs.standard,quotes}
\usetikzlibrary{shapes.geometric}
\newcount\mycount
\usepackage{bm,natbib,soul}

\usepackage{float}
\newfloat{algorithm}{t}{lop}
\floatname{algorithm}{Algorithm}

\TheoremsNumberedThrough 
\ECRepeatTheorems

\EquationsNumberedThrough

\MANUSCRIPTNO{}

\usepackage{afterpage,cleveref}

\usepackage[normalem]{ulem}
\newcommand\redsout{\bgroup\markoverwith{\textcolor{red}{\rule[0.5ex]{2pt}{0.4pt}}}\ULon}
\newcommand\gsout{\bgroup\markoverwith{\textcolor{green}{\rule[0.5ex]{2pt}{0.4pt}}}\ULon}

\usepackage{float}

\usepackage{array}
\newcommand{\PreserveBackslash}[1]{\let\temp=\\#1\let\\=\temp}
\newcolumntype{C}[1]{>{\PreserveBackslash\centering}p{#1}}
\newcolumntype{R}[1]{>{\PreserveBackslash\raggedleft}p{#1}}
\newcolumntype{L}[1]{>{\PreserveBackslash\raggedright}p{#1}}

\RequirePackage[shortlabels]{enumitem}
\usepackage{mathtools} 
\usepackage{enumitem} 
\usepackage{dsfont} 
\usepackage{mathrsfs} 
\usepackage{verbatim}
\usepackage{romannum}
\AtBeginDocument{\pagenumbering{arabic}}

\newcommand{\bu}{{\bm u}}

\newcommand{\bx}{{\bm x}}

\newcommand{\bz}{{\bm z}}

\newcommand{\bbe}{{\bm E}}

\newcommand{\bbl}{{\bm L}}

\newcommand{\bbu}{{\bm U}}
\newcommand{\bbx}{{\bm X}}

\newcommand{\bbz}{{\bm Z}}

\newcommand{\bmu}{{\boldsymbol{\mu}}}

\newcommand{\bdelta}{{\boldsymbol{\delta}}}

\newcommand{\blambda}{{\boldsymbol{\lambda}}}
\newcommand{\bsigma}{{\boldsymbol{\sigma}}}
\newcommand*\bell{\ensuremath{\boldsymbol\ell}}

\newenvironment{itemize*}%
 {\begin{itemize}%
\setlength{\itemsep}{0.5em}%
\setlength{\parskip}{0pt}}%
 {\end{itemize}}

\newcommand{\bepsilon}{{\boldsymbol{\epsilon}}}

\newcommand{\Prb}{\mathbb{P}} 
\newcommand{\Exp}{\mathbb{E}}

\let\N\Natural

\let\R\Real

\newcommand{\bzero}{{\boldsymbol{0}}}

\usepackage{graphicx}
\usepackage{bbm}
\usepackage{tikz,ifthen}
 \usetikzlibrary{math}
 \usetikzlibrary{shapes.misc}
 \usetikzlibrary{shapes.multipart,positioning,patterns,backgrounds}
\newcommand{\ubar}[1]{\text{\b{$#1$}}}

\usepackage{multirow}
\usepackage{array}

\makeatletter
\let\save@mathaccent\mathaccent
\newcommand*\if@single[3]{%
 \setbox0\hbox{${\mathaccent'0362{#1}}^H$}%
 \setbox2\hbox{${\mathaccent'0362{\kern0pt#1}}^H$}%
 \ifdim\ht0=\ht2 #3\else #2\fi
 }
\newcommand*\rel@kern[1]{\kern#1\dimexpr\macc@kerna}
\newcommand*\widebar[1]{\@ifnextchar^{{\wide@bar{#1}{0}}}{\wide@bar{#1}{1}}}
\newcommand*\wide@bar[2]{\if@single{#1}{\wide@bar@{#1}{#2}{1}}{\wide@bar@{#1}{#2}{2}}}
\newcommand*\wide@bar@[3]{%
 \begingroup
 \def\mathaccent##1##2{%
\let\mathaccent\save@mathaccent
\if#32 \let\macc@nucleus\first@char \fi
\setbox\z@\hbox{$\macc@style{\macc@nucleus}_{}$}%
\setbox\tw@\hbox{$\macc@style{\macc@nucleus}{}_{}$}%
\dimen@\wd\tw@
\advance\dimen@-\wd\z@
\divide\dimen@ 3
\@tempdima\wd\tw@
\advance\@tempdima-\scriptspace
\divide\@tempdima 10
\advance\dimen@-\@tempdima
\ifdim\dimen@>\z@ \dimen@0pt\fi
\rel@kern{0.6}\kern-\dimen@
\if#31
 \overline{\rel@kern{-0.6}\kern\dimen@\macc@nucleus\rel@kern{0.4}\kern\dimen@}%
 \advance\dimen@0.4\dimexpr\macc@kerna
 \let\final@kern#2%
 \ifdim\dimen@<\z@ \let\final@kern1\fi
 \if\final@kern1 \kern-\dimen@\fi
\else
 \overline{\rel@kern{-0.6}\kern\dimen@#1}%
\fi
 }%
 \macc@depth\@ne
 \let\math@bgroup\@empty \let\math@egroup\macc@set@skewchar
 \mathsurround\z@ \frozen@everymath{\mathgroup\macc@group\relax}%
 \macc@set@skewchar\relax
 \let\mathaccentV\macc@nested@a
 \if#31
\macc@nested@a\relax111{#1}%
 \else
\def\gobble@till@marker##1\endmarker{}%
\futurelet\first@char\gobble@till@marker#1\endmarker
\ifcat\noexpand\first@char A\else
 \def\first@char{}%
\fi
\macc@nested@a\relax111{\first@char}%
 \fi
 \endgroup
}
\makeatother

\allowdisplaybreaks
\begin{document}


\RUNAUTHOR{Chen, Sturt, Xie}
\RUNTITLE{Blessing of Strategic Customers in Personalized Pricing}

\TITLE{The Blessing of Strategic Customers \\ in Personalized Pricing}

\ARTICLEAUTHORS{
\AUTHOR{Zhi Chen}\AFF{Department of Decisions, Operations and Technology,  The Chinese University of Hong Kong, \EMAIL{zhi.chen@cuhk.edu.hk}}
\AUTHOR{
Bradley Sturt}\AFF{Department of Information and Decision Sciences, University of Illinois  Chicago, \EMAIL{bsturt@uic.edu}}
\AUTHOR{Weijun Xie}
\AFF{H. Milton Stewart School of Industrial and Systems Engineering, Georgia Institute of Technology, \EMAIL{wxie@gatech.edu}}
}

\ABSTRACT{We consider a feature-based personalized pricing problem in which the buyer is \emph{strategic}: given the seller's pricing policy, the buyer can augment the features that they reveal to the seller to obtain a low price for the product. We model the seller's pricing problem as a stochastic program over an infinite-dimensional space of pricing policies where the radii by which the buyer can perturb the features are strictly positive. We establish that the sample average approximation of this problem is \emph{asymptotically consistent}; that is, we prove that the objective value of the sample average approximation converges almost surely to the objective value of the stochastic problem as the number of samples tends to infinity under mild technical assumptions. This consistency guarantee thus shows that incorporating  strategic consumer behavior into a data-driven pricing problem can, in addition to making the pricing problem more realistic, also  help prevent overfitting.\looseness=-1 }
\looseness=-1 


\HISTORY{\today}

\maketitle

\section{Introduction}
Personalized pricing is a problem in revenue management in which a seller is tasked with selecting a price for a single product to offer to a buyer. The buyer's valuation of the product is not directly observed by the seller, but the seller observes features of the buyer (e.g., the buyer's age, credit score, purchase timing) that can be used to predict the buyer's valuation. The goal of the seller is to find a pricing policy that specifies what price to offer to the buyer as a function of the buyer's features to maximize the seller's expected revenue. We model the  personalized pricing problem  as a stochastic program of the form
\begin{equation} \label{prob:trad}
 \begin{aligned}
 &\underset{\pi: \mathcal{X} \to \mathcal{P}}{\text{maximize}} && \Exp\left[ \pi(\bbx) \cdot \mathbb{I} \left \{\pi(\bbx) \le V \right \} \right], 
 \end{aligned}
\end{equation}
where $\pi(\cdot)$ is the pricing policy, $\mathcal{X} \subseteq \R^D$ is the feasible set for the buyer's $D$ features, $\mathcal{P} \subseteq \R$ is the set of feasible prices that can be offered by the seller, and the buyer is represented by the random vector $(\bbx,V)$ consisting of the features $\bbx \in \mathcal{X}$ and product valuation $V \in \mathcal{V} \subseteq \R$. For an introduction to the class of stochastic programs~\eqref{prob:trad}, see  \cite{elmachtoub2021value} and references therein.\looseness=-1

A classical technique in the field of operations research for estimating the objective value of a stochastic program is the \emph{sample average approximation} \citep{shapiro2009lectures}.
This technique consists of two steps: first, we use historical data or simulation to obtain independent samples of the random variables; second, we approximate the stochastic program by replacing the true probability distribution with the empirical probability distribution constructed from the samples. The sample average approximation in application to the personalized pricing problem~\eqref{prob:trad} is denoted by 
\begin{equation} \label{prob:trad_saa}
 \begin{aligned}
 &\underset{\pi: \mathcal{X} \to \mathcal{P}}{\text{maximize}} && \frac{1}{N} \sum_{i=1}^N \pi(\bbx^i) \cdot \mathbb{I} \left \{\pi(\bbx^i) \le V^i \right \},
 \end{aligned}
\end{equation}
where $(\bbx^1,V^1),\ldots,(\bbx^N,V^N)$ are independent samples  drawn from the same  distribution as $ (\bbx,V)$. The objective value of \eqref{prob:trad_saa} is the sample average approximation's estimate of the objective value of \eqref{prob:trad}.\looseness=-1

Unfortunately, if the buyer's features are drawn from a continuous probability distribution, then the objective value of \eqref{prob:trad_saa}  can yield a highly inaccurate estimate of the objective value of \eqref{prob:trad}, {even in the asymptotic regime in which the number of samples tends to infinity} (see, e.g., \citealt{xie2022limitations}).  We illustrate this  by the   following example:
\begin{example} \label{ex:intro}
Suppose that $D = 1$, $\mathcal{P} = \mathcal{X} = \mathcal{V} = \R$, and that $X,V \in \R$ are {independent} random variables that are each uniformly distributed over $[0,1]$. The fact that $X$ and $V$ are independent implies that there exists an optimal pricing policy for \eqref{prob:trad} that offers a constant price for all features; hence,  $\eqref{prob:trad} = \max_{p} \Exp [ p \mathbb{I} \left \{ p \le V \right \}] = \max_{p} \int_{p}^1 p dv= \max_p \{ p - p^2 \} = \sfrac{1}{4}$. However, it follows from the fact that  $X^i \neq X^j$ for all $1 \le i <j \le N$, almost surely,  that there exists an optimal pricing policy for \eqref{prob:trad_saa} that satisfies $\pi(X^i) = V^i$ for all $1 \le i \le N$; hence, $\eqref{prob:trad_saa} = \frac{1}{N} \sum_{i=1}^N V^i \xrightarrow{N \to \infty} \sfrac{1}{2}$ almost surely. 
\end{example}
As illustrated in Example~\ref{ex:intro}, the key issue with the sample average approximation is that \eqref{prob:trad_saa} is optimizing over an {infinite-dimensional} space of pricing policies, which leads \eqref{prob:trad_saa} to output a pricing policy that overfits the samples.

In this note, we study a variant of \eqref{prob:trad} in which the buyer is \emph{strategic}: given the seller's pricing policy, the buyer can augment the features that are revealed to the seller to obtain a low price for the product. 
The radii by which the buyer can perturb their features are positive random variables that are jointly distributed with the buyer's features and valuation.  
Such incorporation of strategic behavior with multi-dimensional features has received significant attention in recent years as an accurate approximation of how customers make decisions (\citealt{tsirtsis2024optimal}), and our model for incorporating strategic behavior into \eqref{prob:trad} can be viewed as a multi-dimensional generalization of intertemporal price discrimination in which customers are strategic in timing their purchases \citep{besbes2015intertemporal}.\looseness=-1

Our main contribution of this note is establishing a \emph{phase transition} for the sample average approximation between the non-strategic and strategic versions of personalized pricing. Specifically, for the strategic version of \eqref{prob:trad}, we prove under mild assumptions that the objective value of the corresponding version of \eqref{prob:trad_saa} is \emph{asymptotically consistent}; that is, the objective value of the sample average approximation converges almost surely to the objective value of the stochastic problem as the number of samples tends to infinity. The asymptotic consistency guarantee holds when optimizing over the infinite-dimensional space of all pricing policies. 
Our result thus shows that accounting for strategic consumer behavior not only makes the personalized pricing problem~\eqref{prob:trad} more realistic: it also 
helps prevent \eqref{prob:trad_saa} from outputting pricing policies that overfit the samples. More generally, our result shows the potential for using strategic consumer behavior as a tool for augmenting the sample average approximation in order to obtain asymptotic consistency guarantees in stochastic programs with infinite-dimensional policy spaces. \looseness=-1

\paragraph{Organization.} In \S\ref{sec:main_result}, we formalize the strategic version of personalized pricing and state our main result, Theorem~\ref{thm:main}. In \S\ref{sec:proof:thm:main}, we present the proof. In \S\ref{sec:numerics}, we illustrate the phase transition implied by Theorem~\ref{thm:main} through a numerical experiment. In \S\ref{sec:intermediary_lemmas}, we present technical details that were omitted from \S\ref{sec:proof:thm:main}. Throughout the paper, random variables are represented by uppercase letters (e.g., $\bbl$) or by lowercase letters with a hat (e.g., $\hat{\nu}_N$). Uppercase letters can also represent a constant or function. Curly upper case letters (e.g., $\mathcal{F},\mathfrak{A}$) represent sets.  Lowercase letters represent a realization of a random variable or a number. A variable is bold if and only if it is a vector.\looseness=-1

\paragraph{Related Literature.} Our consistency result contributes to the literature on data-driven stochastic programming over infinite-dimensional policy spaces.  For the non-strategic version of \eqref{prob:trad} in which the buyer has a one-dimensional feature (i.e, the setting where $D=1$), \cite{xie2022limitations} propose nonparametric ways of imposing restrictions on the infinite-dimensional space of policies in~\eqref{prob:trad_saa} in order to obtain asymptotic consistency guarantees. In recent years, other papers have provided nonparametric techniques for preventing overfitting in stochastic programs with infinite-dimensional policy spaces based on incorporating {adversarial noise} into the sample average approximation \citep{xu2012distributional,shafieezadeh2019regularization,sturt2023nonparametric} and applying kernel smoothing to the samples in the sample average approximation \citep{pflug2016empirical}. We contribute to the aforementioned literature by showing that incorporating {strategic consumer behavior} into the sample average approximation can prevent overfitting in stochastic programs with infinite-dimensional policy spaces.

\section{Problem Setting and Main Result} \label{sec:main_result}
\subsection{Problem Setting}
We consider a strategic variant of the personalized pricing problem defined by a set of allowable prices, $\mathcal{P} \subseteq \R$, and a joint probability distribution for the buyer. The buyer drawn from that joint probability distribution is represented by a random vector $(\bbl,\bbu,V) \in \mathcal{X} \times \mathcal{X} \times \mathcal{V}$ whose support is contained in $\mathcal{F} \times \mathcal{V}$ where $\mathcal{F}\triangleq \left\{(\bell,\bu) \in \mathcal{X} \times \mathcal{X}: \bell \le \bu \right \}$. The random variable $V \in \mathcal{V}$ denotes the buyer's valuation, and the random vectors $\bbl \equiv (L_1,\ldots,L_D) \in \mathcal{X}$ and $\bbu \equiv (U_1,\ldots,U_D) \in \mathcal{X}$ define the random hyperrectangle $[\bbl,\bbu] \equiv \times_{d=1}^D [L_d,U_d]$ from which the buyer chooses the feature vector $\bx \in [\bbl,\bbu]$ that is revealed to the seller.\footnote{The random hyperrectangle $[\bbl,\bbu] \subseteq \R^D$ can be equivalently interpreted as the hyperrectangle $[\bbx - \bbe, \bbx + \bbe]$ where $\bbx \triangleq \frac{1}{2}(\bbl + \bbu)$ is the buyer's random feature vector and $\bbe \triangleq \frac{1}{2} (\bbu - \bbl)$ is the random radii by which the buyer can perturb each of their features.} 

A pricing policy refers to a measurable function $\pi: \mathcal{X} \to \mathcal{P}$ that specifies the price to offer to the buyer based on their revealed feature vector. In other words, if the buyer $(\bbl,\bbu,V)$ reveals $\bx \in [\bbl,\bbu]$ as their feature vector, then $\pi$ will offer the buyer a price of $\pi(\bx) \in \mathcal{P}$. 
Given a pricing policy $\pi$ and a buyer realization $(\bell,\bu,v) \in \mathcal{F} \times \mathcal{V}$, the revenue received by the seller is
\begin{align*}
R^\pi(\bell,\bu,v) &\triangleq \mathbb{I} \left \{\inf_{\bx \in [\bell,\bu]} \pi (\bx) \le v \right \} \cdot \inf_{\bx \in [\bell,\bu]} \pi (\bx ).
\end{align*}
The zero-one indicator $\mathbb{I} \left \{\inf_{\bx \in [\bell,\bu]} \pi (\bx ) \le v \right \}$ implies that the buyer realization $(\bell,\bu,v)$ will purchase the product if and only if there exists a feature vector $\bx \in [\bell,\bu]$ that the buyer can reveal to the seller that generates a price $\pi(\bx)$ that is less than or equal to the buyer's valuation $v$. If the buyer purchases the product, then the buyer will maximize their utility by revealing the feature vector $\bx \in [\bell,\bu]$ that leads to the minimum price. The goal of the seller is to choose a pricing policy that maximizes the seller's expected revenue, which is captured by the following stochastic program:\looseness=-1
\begin{align}
\nu^* \triangleq \sup_{\pi: \mathcal{X} \to \mathcal{P}} \left \{ J^*(\pi) \triangleq \Exp \left[R^\pi \left(\bbl,\bbu,V\right) \right] \right \}. 
 \tag{OPT} \label{prob:opt}
\end{align}

\subsection{Main Result}
Our main result establishes the asymptotic consistency of the objective value of the sample average approximation with respect to the objective value $\nu^*$ of \eqref{prob:opt}. We make the following  assumptions:\looseness=-1
\begin{assumption} \label{ass:iid}
$(\bbl,\bbu,V),(\bbl^1,\bbu^1,V^1),(\bbl^2,\bbu^2,V^2),\ldots$ is a sequence of independent and identically distributed 
random vectors with support contained in $\mathcal{F} \times \mathcal{V}$. 
\end{assumption}
\begin{assumption} \label{ass:bounded}
$\mathcal{X} = [0,1]^D$, $\mathcal{V} = [0,1]$, and $\mathcal{P} \subseteq [0,1]$  is closed and nonempty.
\end{assumption}
\begin{assumption} \label{ass:continuous}
$(\bbl,\bbu) \in \R^{2D}$ is a continuous random vector with probability   density function $g: \R^{2D} \to [0,\infty)$.\looseness=-1  
\end{assumption}
Assumption~\ref{ass:iid} is a standard assumption in the stochastic programming literature and allows the random variables $\bbl,\bbu,V$ to be correlated with one another. 
Assumption~\ref{ass:bounded} is a boundedness assumption on the feature vector and product valuations, which is imposed to simplify our proofs, and a regularity assumption on the set of  prices $\mathcal{P}$. Similarly as \cite{xie2022limitations}, 
 our main result can be extended to the case of features and  valuations with unbounded supports under mild probabilistic assumptions using standard truncation arguments. We also note that the regularity assumption on $\mathcal{P}$ is satisfied if there exist prices $p_1 < \cdots < p_K$ such that $\mathcal{P} = \{p_1,\ldots,p_K\}$. Assumption~\ref{ass:continuous} says that $(\bbl,\bbu)$ is drawn from a continuous probability distribution with respect to $\R^{2D}$, and thus implies that $\Prb((\bbl,\bbu) \in \mathcal{Z}) = 0$ for any measurable set $\mathcal{Z} \subseteq \R^{2D}$ with zero Lebesgue measure. We note that Assumption~\ref{ass:continuous} is satisfied if $\bbl = \bbx - \bepsilon$ and $\bbu = \bbx + \bepsilon$ for fixed radii $\bepsilon > \bzero$ and a random vector $\bbx \in \R^{D}$ that satisfies $\Prb(\bbx \in [\bepsilon, \mathbf{1} - \bepsilon]) = 1$ and is absolutely continuous with respect to the Lebesgue measure on $\R^{D}$. We show in Appendix~\ref{sec:discussion:continuous} that  Assumption~\ref{ass:continuous} is a necessary condition for asymptotic consistency, in the sense that relaxing this assumption can invalidate our main result.

Our main result  is the following:
\begin{theorem} \label{thm:main}
Suppose Assumptions~\ref{ass:iid}, \ref{ass:bounded}, and \ref{ass:continuous} hold, and let 
\begin{align}
 \widehat{\nu}_N \triangleq \sup_{\pi: \mathcal{X} \to \mathcal{P}} \left \{ \widehat{J}_N(\pi) \triangleq \frac{1}{N} \sum \limits_{i=1}^N R^{\pi}\left(\bbl^i,\bbu^i,V^i \right) \right \} \tag{SAA} \label{prob:saa}
\end{align}
for each integer $N > 0$. Then $\lim\limits_{N \to \infty} \widehat{\nu}_N = \nu^*$ almost surely. 
\end{theorem}

\section{Proof of Theorem~\ref{thm:main}}\label{sec:proof:thm:main}
We assume throughout this section that Assumptions~\ref{ass:iid}--\ref{ass:continuous} hold. We additionally assume  that:\looseness=-1
\begin{assumption} \label{ass:discrete}
There exist $0 < p_1 < \cdots < p_K \le 1$ such that $\mathcal{P} = \{p_1,\ldots,p_K \}$.
\end{assumption}
This assumption is unnecessary for establishing Theorem~\ref{thm:main} but simplifies the presentation of its proof. 
 In \S\ref{sec:proof:remove_ass}, we prove that Theorem~\ref{thm:main} holds in the absence of Assumption~\ref{ass:discrete}. 

It follows from \citet[Equation (5.6)]{shapiro2009lectures} that \eqref{prob:saa} is an asymptotically {optimistic} estimate of \eqref{prob:opt}, in the sense that $\liminf_{N \to \infty} \widehat{\nu}_N \ge \nu^*$ 
 almost surely.  It remains to prove the other direction of that inequality. To this end, we introduce a grid-based partitioning of the feature space. Specifically, for each integer $S > 0$, we consider partitioning $[0,1]^D$ into hypercubes of width $\sfrac{1}{S}$, where the hypercubes are indexed by vectors $\bsigma \equiv (\sigma_1,\ldots,\sigma_D) \in \{1,\ldots,S\}^D$. The set of feature vectors in the hypercube corresponding to $\bsigma$ is denoted by 
$\mathcal{H}_S(\bsigma) \triangleq \times_{d=1}^D \mathcal{H}_{S,d}(\sigma_d)$
where 
\begin{align*}
\mathcal{H}_{S,d}(\sigma_d) \triangleq \begin{cases}
\left[\frac{\sigma_d-1}{S}, \frac{\sigma_d}{S} \right) &\text{if } d \in \{1,\ldots,D-1\} \\
\left[\frac{\sigma_d-1}{S}, \frac{\sigma_d}{S} \right] &\text{if } d = D. 
\end{cases}
\end{align*}
We readily observe that $\{ \mathcal{H}_S(\bsigma): \bsigma \in \{1,\ldots,S\}^D \}$ is a collection of disjoint sets whose union is equal to $[0,1]^D$. This implies that every feature vector $\bx \in [0,1]^D$ is contained in exactly one hypercube. 
For notational convenience, we denote the indices of the hypercube that contains feature vector $\bx \in [0,1]^D$ by $\bsigma_S(\bx) \equiv (\sigma_{S,1}(x_1),\ldots,\sigma_{S,D}(x_D))\in \{1,\ldots,S\}^D$; that is, $\bsigma_S(\bx)$ is defined as the unique vector of indices that satisfies $\bx \in \mathcal{H}_S(\bsigma_S(\bx))$. 

Let the set of all pricing policies that are piecewise constant over the grid-based partitioning of $[0,1]^D$ into hypercubes of width $\sfrac{1}{S}$ be denoted by
\begin{align*}
\Pi_{S}\triangleq \left \{\pi \in \Pi \; \middle \vert \; 
 \text{if } \bsigma \in \{1,\ldots,S \}^D \text{ and } 
 \bx,\bx' \in \mathcal{H}_S(\bsigma), \text{ then } \pi(\bx) = \pi(\bx') \right \}.
\end{align*}
It follows from the above definitions that the restricted space of pricing policies $\Pi_S$ has a finite number of elements. More specifically, we observe from Assumption~\ref{ass:discrete} that for each integer $S > 0$, the restricted space of pricing policies has a cardinality of $|\Pi_S| = K^{S^D}$. It follows from this observation that the objective value of the sample average approximation is a consistent estimator of the objective value of the stochastic problem when optimizing over the restricted space $\Pi_S$, in the sense that the following line holds for every integer $S > 0$:\footnote{More precisely, \eqref{line:discrete:uniform} follows from Assumption~\ref{ass:iid}, from the fact that $|\Pi_S| < \infty$, and from the fact that each $\pi \in \Pi_S$ is measurable with respect to the Lebesgue measure over $[0,1]^D$, which implies the function $(\bell,\bu,v) \in \mathcal{F} \times [0,1] \mapsto R^\pi(\bell,\bu,v) \in [0,1]$ is integrable \citep[Proposition 5.2 and Theorem 7.48]{shapiro2009lectures}.}
 \begin{align}
 \lim \limits_{N \to \infty} \sup \limits_{\pi \in \Pi_S} \widehat{J}_N(\pi)= \sup \limits_{\pi \in \Pi_S} J^*(\pi) \quad \text{almost surely}.\label{line:discrete:uniform}
\end{align}

The remainder of the proof of Theorem~\ref{thm:main} consists of showing for every sufficiently large integer $S > 0$ and for every pricing policy $\pi: [0,1]^D \to \mathcal{P}$ that there exists a pricing policy in the restricted space $\Pi_S$ yielding  approximately the same objective value with respect to the sample average approximation. 
To show this, let $\Pi \triangleq \{\pi: [0,1]^D \to \mathcal{P} \}$ denote the space of all measurable pricing policies, and let $T_S: \Pi \to \Pi_S$ be an operator defined as 
\begin{align*}
 (T_S \circ \pi)(\bx) \triangleq \min_{\bx' \in \mathcal{H}_S(\bsigma_S(\bx))}\pi(\bx') \quad \forall \pi \in \Pi, \bx \in [0,1]^D.
\end{align*} 
The pricing policy $T_S \circ \pi$ can be interpreted as a `rounded' version of pricing policy $\pi \in \Pi$ that, for each $\bsigma \in \{1,\ldots,S\}^D$, offers the constant price $\min_{\bx' \in \mathcal{H}_S(\bsigma)} \pi(\bx')$ to all feature vectors in the hypercube $\mathcal{H}_S(\bsigma)$. 
The following relates the asymptotic performance of $T_S \circ \pi$ relative to $\pi$.
\begin{proposition}
\label{prop:main}
 $ \limsup \limits_{S \to \infty} \limsup \limits_{N \to \infty} \sup \limits_{\pi \in \Pi} \left \{ \widehat{J}_N\left(\pi \right) - \widehat{J}_N\left(T_S \circ\pi \right) \right \} = 0 $ almost surely.
\end{proposition}
The proof of Proposition~\ref{prop:main} is found in \S\ref{sec:proof:prop:main}. Combining Proposition~\ref{prop:main} and \eqref{line:discrete:uniform}, we have
\begin{align*}
\limsup_{N \to \infty} \widehat{\nu}_N &= \limsup_{N \to \infty} \sup_{\pi \in \Pi} \widehat{J}_N(\pi) = \limsup_{S \to \infty} \limsup_{N \to \infty} \sup_{\pi \in \Pi} \left \{ \widehat{J}_N(\pi) - \widehat{J}_N(T_S \circ \pi) + \widehat{J}_N(T_S \circ \pi) \right \} \\
&\le \limsup_{S \to \infty} \limsup_{N \to \infty} \sup_{\pi \in \Pi} \left \{ \widehat{J}_N(\pi) - \widehat{J}_N(T_S \circ \pi) \right \} + \limsup_{S \to \infty} \limsup_{N \to \infty} \sup_{\pi \in \Pi_S} \widehat{J}_N(\pi) \\
&= \limsup_{S \to \infty} \lim_{N \to \infty} \sup_{\pi \in \Pi_S} \widehat{J}_N(\pi) \quad \text{almost surely} \\
&= \limsup_{S \to \infty} \sup \limits_{\pi \in \Pi_S} J^*(\pi) \quad \text{almost surely} \\
&\le \nu^*.
\end{align*}
The first and second equalities and the first inequality follow from algebra. The third equality follows from Proposition~\ref{prop:main}. The fourth equality follows from \eqref{line:discrete:uniform}. The final inequality follows from the fact that $\Pi_S \subseteq \Pi$ for all $S > 0$. Combining the above with the fact that $\liminf_{N \to \infty} \widehat{\nu}_N \ge \nu^*$ almost surely completes the proof of Theorem~\ref{thm:main}. 

\subsection{Proof of Proposition~\ref{prop:main}} \label{sec:proof:prop:main}
Given any integer $S > 0$, we define a bucket as any tuple of the form $(\blambda, \bmu) \in \mathcal{B}_S$ where $$
 \mathcal{B}_S \triangleq \left\{(\blambda,\bmu) \in \{1,\ldots,S\}^D \times \{1,\ldots,S\}^D: \blambda \le \bmu \right\}.$$
We say henceforth that a buyer realization $(\bell,\bu,v) \in \mathcal{F} \times [0,1]$ is mapped to bucket $(\blambda,\bmu) \in \mathcal{B}_S$ if and only if $ \blambda = \bsigma_S(\bell)$ and $\bmu = \bsigma_S(\bu)$. That is,  buyer realization $(\bell,\bu,v)$ is mapped to bucket $(\blambda,\bmu)$ if and only if $\bell$ is an element of the hypercube $\mathcal{H}_S(\blambda)$ and $\bu$ is an element of the hypercube $\mathcal{H}_S(\bmu)$. It follows that every buyer realization is mapped to exactly one bucket.\looseness=-1

Let the buyer realizations that are mapped to bucket $(\blambda,\bmu)$ be denoted by 
$$\mathcal{C}_S(\blambda,\bmu) \triangleq \left \{ (\bell,\bu,v) \in \mathcal{F} \times [0,1]: \blambda = \bsigma_S(\bell), \bmu = \bsigma_S(\bu)\right \}.$$ 
For all integers $S,N >0$ and buckets $(\blambda,\bmu) \in \mathcal{B}_S$, let the subset of the $N$ random buyers that are mapped to bucket $(\blambda,\bmu)$ be denoted by $$\widehat{\mathcal{N}}_{S,N}(\blambda,\bmu) \triangleq \{i \in \{1,\ldots,N\}: (\bbl^i,\bbu^i,V^i) \in \mathcal{C}_S(\blambda,\bmu)\}. $$
For each pricing policy $\pi \in \Pi$, let the objective value of the sample average approximation corresponding to $\pi$ for the random buyers that are mapped to  bucket $(\blambda,\bmu)$ be denoted by
\begin{align*}
\widehat{J}_{S,N}(\pi;\blambda,\bmu) \triangleq \frac{1}{ | \widehat{\mathcal{N}}_{S,N}(\blambda,\bmu) |} \sum_{i \in \widehat{\mathcal{N}}_{S,N}(\blambda,\bmu)} R^\pi \left(\bbl^i,\bbu^i,V^i \right), 
\end{align*}
where we use the convention that $\sfrac{1}{0} \times 0 = 0$. Finally, let the probability that a random buyer is mapped to bucket $(\blambda,\bmu)$ be denoted by 
$p_{S}(\blambda,\bmu) \triangleq \Prb \left( \bsigma_S(\bbl) = \blambda , \bsigma_S(\bbu) = \bmu \right).$

It follows from the above notations that 
\begin{align}
&\limsup_{S \to \infty} \limsup_{N \to \infty} \sup_{\pi \in \Pi} \left \{ \widehat{J}_N\left(\pi \right) - \widehat{J}_N\left(T_S \circ\pi \right) \right \} \notag \\
= & \limsup_{S \to \infty} \limsup_{N \to \infty} \sup_{\pi \in \Pi} \left \{ \sum_{(\blambda,\bmu) \in \mathcal{B}_S} \frac{ | \widehat{\mathcal{N}}_{S,N}(\blambda,\bmu) |}{N} \left(\widehat{J}_{S,N}(\pi; \blambda,\bmu) - \widehat{J}_{S,N}(T_S \circ \pi; \blambda,\bmu) \right) \right \} \notag \\
\le & \limsup_{S \to \infty} \limsup_{N \to \infty} \sup_{\pi \in \Pi} \left \{ \sum_{(\blambda,\bmu) \in \mathcal{B}_S} \frac{ | \widehat{\mathcal{N}}_{S,N}(\blambda,\bmu) |}{N} \sup_{(\bell,\bu,v) \in \mathcal{C}_S(\blambda,\bmu)}\left \{R^\pi(\bell,\bu,v) - R^{T_S \circ \pi}(\bell,\bu,v) \right \}\right \} \notag \\
= & \limsup_{S \to \infty} \sup_{\pi \in \Pi} \left \{ \sum_{(\blambda,\bmu) \in \mathcal{B}_S} p_S(\blambda,\bmu) \sup_{(\bell,\bu,v) \in \mathcal{C}_S(\blambda,\bmu)}\left \{R^\pi(\bell,\bu,v) - R^{T_S \circ \pi}(\bell,\bu,v) \right \} \right \} \quad \text{almost surely}, \label{line:lln_applied}
\end{align}
where the first equality and the first inequality follow from algebra and the second equality follows from the uniform law of large numbers.\footnote{The uniform law of large numbers \citep[Theorem 7.48]{shapiro2009lectures} says for every $S > 0$ and $\delta > 0$, there exists a random variable $\bar{N}$ satisfying $\bar{N} < \infty$ almost surely such that for all $N \ge \bar{N}$ and all $(\blambda,\bmu) \in \mathcal{B}_S$, we have $ p_S(\blambda,\bmu) - \delta\leq | \widehat{\mathcal{N}}_{S,N}(\blambda,\bmu) |/ N \le p_S(\blambda,\bmu) + \delta $. Hence, line~\eqref{line:lln_applied} follows from choosing $\delta > 0$ to be arbitrarily small. }

The remainder 
of the proof of Proposition~\ref{prop:main} consists of showing, roughly speaking, that the inner-most supremum in line~\eqref{line:lln_applied} is nonpositive for \emph{most} buckets. More precisely, for each integer $S > 0$ and pricing policy $\pi \in \Pi$, let the buckets for which the inner-most supremum is positive be denoted by\looseness=-1
\begin{align*}
 \mathcal{A}^\pi_{S} \triangleq \left \{ (\blambda, \bmu) \in \mathcal{B}_S: \sup_{(\bell,\bu,v) \in \mathcal{C}_S(\blambda,\bmu)}\left \{R^\pi(\bell,\bu,v) - R^{T_S \circ \pi}(\bell,\bu,v) \right \} > 0 \right \}. 
\end{align*}
The following proposition shows that the probability of a random buyer mapping to a bucket in $\mathcal{A}^\pi_S$ converges to zero as $S$ tends to infinity, uniformly over all pricing policies. 
\begin{proposition}\label{prop:pi}
 $\limsup\limits_{S \to \infty} \sup\limits_{\pi \in \Pi} \sum\limits_{(\blambda,\bmu) \in \mathcal{A}^\pi_S} p_S(\blambda,\bmu) = 0$. 
\end{proposition}
The proof of Proposition~\ref{prop:pi} is found in \S\ref{sec:proof:prop:pi}. Equipped with the above proposition, we observe that
\begin{align*}
 \eqref{line:lln_applied} &\le \limsup_{S \to \infty} \sup_{\pi \in \Pi} \left \{ \sum_{(\blambda,\bmu) \in \mathcal{B}_S \setminus \mathcal{A}_S^\pi} p_S(\blambda,\bmu) \sup_{(\bell,\bu,v) \in \mathcal{C}_S(\blambda,\bmu)}\left \{R^\pi(\bell,\bu,v) - R^{T_S \circ \pi}(\bell,\bu,v) \right \} \right \}\\
 &\quad + \limsup_{S \to \infty} \sup_{\pi \in \Pi} \left \{ \sum_{(\blambda,\bmu) \in \mathcal{A}_S^\pi} p_S(\blambda,\bmu) \sup_{(\bell,\bu,v) \in \mathcal{C}_S(\blambda,\bmu)}\left \{R^\pi(\bell,\bu,v) - R^{T_S \circ \pi}(\bell,\bu,v) \right \} \right \}\\
 &\le 0 + \limsup_{S \to \infty} \sup_{\pi \in \Pi} \left \{ \sum_{(\blambda,\bmu) \in \mathcal{A}_S^\pi} p_S(\blambda,\bmu) \cdot 1 \right \}=0,
\end{align*} 
where the first inequality follows from algebra, the second inequality follows from the fact that $\sup_{(\bell,\bu,v) \in \mathcal{C}_S(\blambda,\bmu)}\left \{R^\pi(\bell,\bu,v) - R^{T_S \circ \pi}(\bell,\bu,v) \right \} \le 0$ for all $(\blambda,\bmu) \in \mathcal{B}_S \setminus \mathcal{A}^\pi_S$ and from the fact that $R^\pi(\bell,\bu,v) - R^{T_S \circ \pi}(\bell,\bu,v)\le 1$ for all $(\bell,\bu,v) \in \mathcal{F} \times [0,1]$, and the equality follows from Proposition~\ref{prop:pi}. That concludes the proof of Proposition~\ref{prop:main}.

\subsection{Proof of Proposition~\ref{prop:pi}} \label{sec:proof:prop:pi}
Consider any arbitrary $S > 1$ and pricing policy $\pi \in \Pi$. 

We begin by introducing a necessary condition (Lemma~\ref{lem:characterization_customers}) that must be satisfied by every bucket that is contained in $\mathcal{A}_S^\pi$. To state this necessary condition, we use the following terminology. For each bucket $(\blambda,\bmu) \in \mathcal{B}_S$, we let the hyperrectangle formed by the set of hypercubes $\bsigma \in \{1,\ldots,S\}^D$ that satisfy $\blambda \le \bsigma \le \bmu$ be denoted by $\textsc{Rect}(\blambda,\bmu) \triangleq \times_{d=1}^D \left \{\lambda_d ,\ldots,\mu_d \right \}$. 
Moreover, for each bucket $(\blambda,\bmu) \in \mathcal{F}_{S}$ and each vector $\bdelta \in \{-1,0,1\}^D$, we define the $\bdelta$-{face} of $\textsc{Rect}(\blambda,\bmu)$ as 
\begin{align*}
\textsc{Face}(\blambda,\bmu,\bdelta) \triangleq \left \{ \bsigma: \; \begin{aligned}
\sigma_d &= \lambda_d && \text{ for all } d \in \{1,\ldots,D\} \text{ such that } \delta_d = -1 \\ 
\sigma_d &= \mu_d && \text{ for all } d \in \{1,\ldots,D\} \text{ such that } \delta_d = 1 \\ 
\sigma_d &\in \{\lambda_d + 1,\ldots,\mu_d - 1\} &&\text{ for all } d \in \{1,\ldots,D\} \text{ such that } \delta_d = 0 
\end{aligned}\right \}.
\end{align*}
It is easy to verify for each bucket $(\blambda,\bmu) \in \mathcal{B}_S$ that $\{\textsc{Face}(\blambda,\bmu,\bdelta): \bdelta \in \{-1,0,1\}^D\}$ is a collection of sets whose union is equal to $\textsc{Rect}(\blambda,\bmu)$. In Figures~\ref{fig:geometry}a and \ref{fig:geometry}b, we visualize hyperrectangles and $\bdelta$-faces. We now present our necessary condition for a bucket to be an element of $\mathcal{A}_S^\pi$.\looseness=-1
\begin{figure}[t]
\centering
\FIGURE{
\begin{minipage}{\linewidth}
\centering
\subfloat[{}]{
\includegraphics[width=0.38\textwidth]{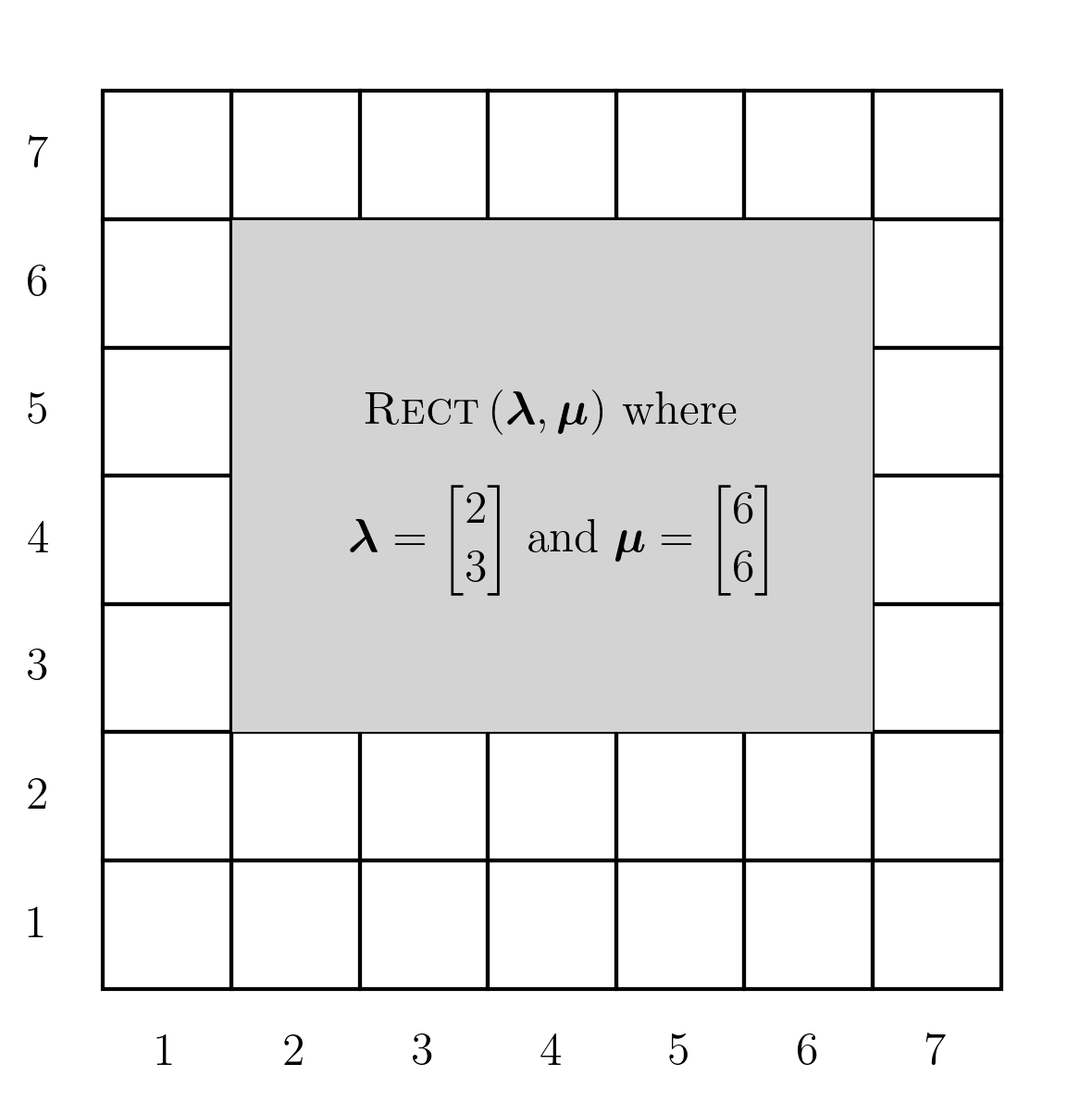}
} \quad
\subfloat[]{
\includegraphics[width=0.435\textwidth]{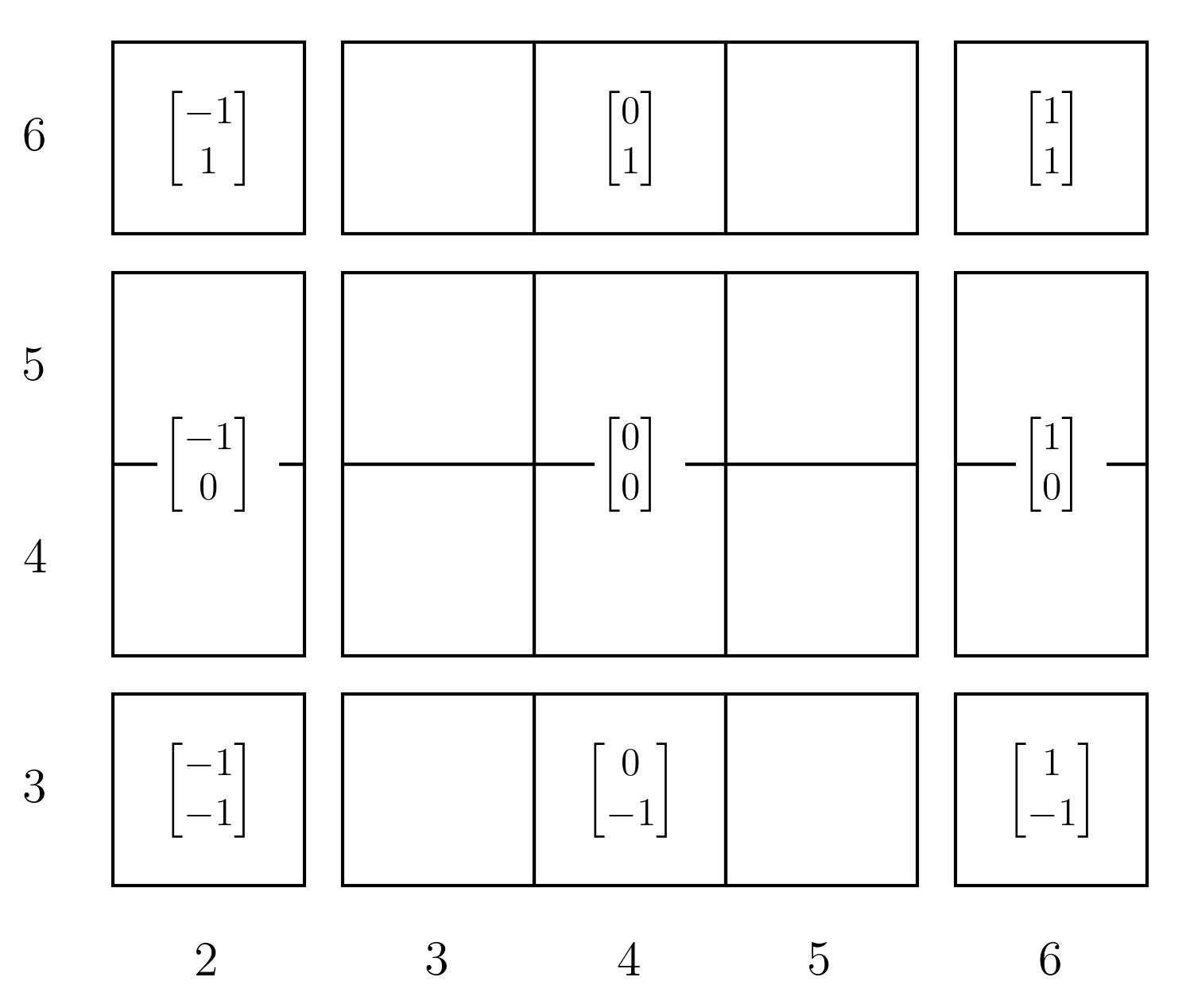}
}
\end{minipage}
}
{Visualizations from \S\ref{sec:proof:prop:pi}\\[-2em]\label{fig:geometry}} 
{ \emph{(a)} Visualization of grid-based partitioning of the feature space $[0,1]^2$ into hypercubes of width $\frac{1}{7}$. \emph{(b)} Visualization of $\textsc{Face}\left((2,3),\left(6,6\right),\bdelta \right)$ for each $\bdelta \in \{-1,0,1\}^2$. Note that the union of the $\bdelta$-faces in (\emph{b}) is equal to the hyperrectangle $\textsc{Rect}((2,3),(6,6))$ which is shown shaded in (\emph{a}).} 
\end{figure}
\begin{lemma} \label{lem:characterization_customers}
If $(\blambda,\bmu) \in \mathcal{A}_{S}^\pi$, then $\min \limits_{\bsigma \in \textsc{Face}(\blambda,\bmu,\bzero)} \min \limits_{\bx \in \mathcal{H}_S(\bsigma)} \pi(\bx) > \min \limits_{\bdelta \in \{-1,0,1\}^D: \bdelta \neq \bzero} \min \limits_{\bsigma \in \textsc{Face}(\blambda,\bmu,\bdelta)} \min \limits_{\bx \in \mathcal{H}_S(\bsigma)} \pi(\bx). $
\end{lemma}

The proof of the above lemma (as well as the proof of each subsequent lemma in \S\ref{sec:proof:thm:main}) can be found in \S\ref{sec:proof:lemmas}. In words, the above lemma says that if a bucket satisfies $(\blambda,\bmu) \in \mathcal{A}^\pi_S$, then there must exist a feature vector in one of the hypercubes in the boundary of the hyperrectangle corresponding to the bucket, $\bx \in \cup_{\bdelta \in \{-1,0,1\}^D: \bdelta \neq \bzero} \cup_{\bsigma \in \textsc{Face}(\blambda,\bmu,\bdelta)} \mathcal{H}_S(\bsigma) $, for which the pricing policy offers a price, $\pi(\bx)$, that is strictly less than the price offered by the pricing policy for all feature vectors in the hypercubes in the interior of the hyperrectangle corresponding to the bucket. For every integer $M \in \{1,\ldots,S-1\}$, it follows from Lemma~\ref{lem:characterization_customers} that 
\begin{align}
&\mathcal{A}^\pi_S\subseteq \bigcup_{\bdelta \in \{-1,0,1\}^D: \; \bdelta \neq \bzero} \left \{ (\blambda,\bmu) \in \mathcal{B}_{S}: \; \min_{\bsigma \in \textsc{Face}(\blambda,\bmu,\bzero)} \min_{\bx \in \mathcal{H}_S(\bsigma)} \pi(\bx) > \min \limits_{\bsigma \in \textsc{Face}(\blambda,\bmu,\bdelta)} \min \limits_{\bx \in \mathcal{H}_S(\bsigma)} \pi(\bx) \right \} \notag \\
& \subseteq \mathcal{B}^{<}_{S,M} \cup \bigcup_{\bdelta \in \{-1,0,1\}^D: \bdelta \neq \bzero} \underbrace{ \left \{ (\blambda,\bmu) \in \mathcal{B}_{S,M}^{\ge}: \min_{\bsigma \in \textsc{Face}(\blambda,\bmu,\bzero)} \min_{\bx \in \mathcal{H}_S(\bsigma)} \pi(\bx) > \min \limits_{\bsigma \in \textsc{Face}(\blambda,\bmu,\bdelta)} \min \limits_{\bx \in \mathcal{H}_S(\bsigma)} \pi(\bx) \right \}}_{\triangleq \mathcal{A}^\pi_{S,M,\bdelta }} \label{line:A_pi_S_supset}
\end{align}
where the sets $\mathcal{B}^<_{S,M}$ and $\mathcal{B}^{\ge}_{S,M}$ in \eqref{line:A_pi_S_supset} are defined below as 
\begin{align*}
\mathcal{B}_{S,M}^{<} \triangleq\left \{ (\blambda,\bmu) \in \mathcal{B}_S: \min_{d \in \{1,\ldots,D\}} \left \{ \mu_d - \lambda_d \right \} < M \right \} \text{ and } \mathcal{B}_{S,M}^{\ge} \triangleq\left \{ (\blambda,\bmu) \in \mathcal{B}_S: \min_{d \in \{1,\ldots,D\}} \left \{ \mu_d - \lambda_d \right \} \ge M \right \}.
\end{align*}
For each integer $M \in \{1,\ldots,S-1\}$, it follows from algebra that
\begin{align}
\left| \mathcal{B}^<_{S,M} \right| &\le \sum_{m=0}^{M-1} \sum_{d=1}^D \left| \left \{ (\blambda,\bmu) \in \mathcal{B}_S: \mu_d - \lambda_d = m \right \} \right| \le M D S^{2D-1}. \label{line:A_pi_S_supset:1}
\end{align}
Moreover, the following proposition establishes a conservative upper bound on the cardinality of $\mathcal{A}^\pi_{S,M,\bdelta}$ for every integer $M \in \{1,\ldots,S-1\}$ and every vector $\bdelta \in \{-1,0,1\}^D$ that satisfies $\bdelta \neq \bzero$:
\begin{proposition} \label{prop:bound_A_S_delta_omega}
 $| \mathcal{A}^\pi_{S,M,\bdelta}| \le \frac{ 2D K S^{2D}}{M}$.
\end{proposition}
Since $\pi \in \Pi$ was chosen arbitrarily, we observe from the above reasoning that the following holds for every arbitrary $S > 2$:
\begin{align}
\sup_{\pi \in \Pi} \frac{\left| \mathcal{A}^\pi_S \right| }{S^{2D}}  &\le \sup_{\pi \in \Pi} \left \{\frac{1}{S^{2D}} \cdot \min_{M \in \{1,\ldots,S-1\}} \left \{ \left| \mathcal{B}^<_{S,M} \right| + \sum_{\bdelta \in \{-1,0,1\}^D: \bdelta \neq \bzero} \left| \mathcal{A}^\pi_{S,M,\bdelta} \right| \right \} \right \} \notag \\
&\le \frac{1}{S^{2D}} \cdot \min_{M \in \{1,\ldots,S-1\}} \left \{ M D S^{2D-1} + \sum_{\bdelta \in \{-1,0,1\}^D: \bdelta \neq \bzero}\frac{2 D K S^{2D}}{M} \right \}  \notag \\
&\le  \underbrace{ \frac{ D}{S} \left(  \left \lceil \sqrt{S} \right \rceil  + \frac{ \left(3^D - 1 \right) 2 K S}{ \left \lceil \sqrt{S} \right \rceil} \right )}_{\triangleq \beta_S}.
\label{line:beta_S}
\end{align}
Indeed, the first inequality follows from \eqref{line:A_pi_S_supset}, the second inequality follows from \eqref{line:A_pi_S_supset:1} and Proposition~\ref{prop:bound_A_S_delta_omega}, the third inequality follows from setting  $M = \left \lceil \sqrt{S} \right \rceil$. Since $S$  was chosen arbitrarily, we have\looseness=-1
\begin{align}
\limsup_{S \to \infty} \sup_{\pi \in \Pi} \frac{\left| \mathcal{A}^\pi_S \right| }{S^{2D}}  \le \lim_{S \to \infty} \beta_S = 0, \label{line:A_gets_small}
\end{align}
where the inequality follows from \eqref{line:beta_S} and the equality follows from the definition of $\beta_S$. 

To interpret the upper bound from line~\eqref{line:A_gets_small}, we recall that the number of buckets satisfies $|\mathcal{B}_S| = S^{2D}$. Line~\eqref{line:A_gets_small} can thus be interpreted as showing that the proportion $\frac{| \mathcal{A}^\pi_S|}{|\mathcal{B}_S|}$ of buckets that are in the set $\mathcal{A}_S^\pi$ converges to zero as $S$ tends to infinity, uniformly over all $\pi \in \Pi$. To conclude the proof of Proposition~\ref{prop:pi}, we use Assumption~\ref{ass:continuous} to show that the probability of a random buyer being contained in any vanishingly small proportion of buckets converges to zero as $S$ tends to infinity. This result is shown in the following lemma. 
\begin{lemma} \label{lem:proportion_bucket}
Given any sequence $\alpha_1,\alpha_2,\ldots \in (0,1)$ that satisfies $\lim\limits_{S \to \infty} \alpha_S = 0$, we have $$\lim \limits_{S \to \infty} \sup \limits_{\mathcal{A}_S \subseteq \mathcal{B}_S: \frac{| \mathcal{A}_S|}{| \mathcal{B}_S|} \le \alpha_S } \sum_{(\blambda,\bmu) \in \mathcal{A}_S} p_S(\blambda,\bmu) = 0.$$
\end{lemma}
We conclude that
\begin{align*}
\limsup_{S \to \infty} \sup_{\pi \in \Pi} \sum_{(\blambda,\bmu) \in \mathcal{A}^\pi_S} p_S(\blambda,\bmu) \le \limsup_{S \to \infty} \sup_{\mathcal{A}_S \subseteq \mathcal{B}_S: \frac{| \mathcal{A}_S|}{| \mathcal{B}_S|} \le \beta_S} \sum_{(\blambda,\bmu) \in \mathcal{A}_S} p_S(\blambda,\bmu) = 0,
\end{align*}
where the first inequality follows from line~\eqref{line:beta_S}  and the equality follows from Lemma~\ref{lem:proportion_bucket} and line~\eqref{line:A_gets_small}. This completes our proof of Proposition~\ref{prop:pi}. 
\subsection{Proof of Proposition~\ref{prop:bound_A_S_delta_omega}} \label{sec:proof:lines}

Consider any arbitrary $S > 1$, pricing policy $\pi \in \Pi$, vector $\bdelta \in \{-1,0,1\}^D$ that satisfies $\bdelta \neq \bzero$, and integer $M \in \{1,\ldots,S-1\}$. We say henceforth that a bucket $(\blambda,\bmu) \in \mathcal{B}^{\ge}_{S,M}$ is a $\bdelta$-anchor if and only if there exists $d \in \{1,\ldots,D\}$ such that $\delta_d \neq 0$, $\lambda_d = 1$ if $\delta_d = 1$, and $\mu_d = S$ if $\delta_d = -1$. For convenience, we denote the set of all $\bdelta$-anchors by 
\begin{align*}
\mathfrak{A}_{S,M}(\bdelta) \triangleq \bigcup_{d \in \{1,\ldots,D\}: \delta_d = 1} \left \{ (\blambda,\bmu) \in \mathcal{B}^{\ge}_{S,M}: \lambda_d = 1 \right \}  \cup \bigcup_{d \in \{1,\ldots,D\}: \delta_d = -1} \left \{ (\blambda,\bmu) \in \mathcal{B}^{\ge}_{S,M}: \mu_d = S \right \} .
\end{align*}
Furthermore, we say henceforth that $(\blambda,\bmu) \in \mathcal{B}^{\ge}_{S,M}$ is contained in the $\bdelta$-{line} emanating from $(\bar{\blambda},\bar{\bmu}) \in \mathfrak{A}_{S,M}(\bdelta)$ if and only if there exists an integer $\iota \ge 0$ such that $(\blambda,\bmu) = (\bar{\blambda},\bar{\bmu}) + \iota (\bdelta,\bdelta)$. For convenience, we denote the 
$\bdelta$-line emanating from $(\bar{\blambda},\bar{\bmu}) \in \mathfrak{A}_{S,M}(\bdelta)$ by $$\mathcal{L}_{S,M}(\bar{\blambda},\bar{\bmu},\bdelta) \triangleq \{ (\blambda,\bmu) \in \mathcal{B}^{\ge}_{S,M}: \exists\text{ integer } \iota \ge 0 \text{ such that } (\blambda,\bmu) = (\bar{\blambda},\bar{\bmu}) + \iota (\bdelta,\bdelta)\}.$$
 
\begin{figure}[t]
\centering
\FIGURE{%
\begin{minipage}{\linewidth}
\centering
\subfloat[]{
\includegraphics[width=0.4\textwidth]{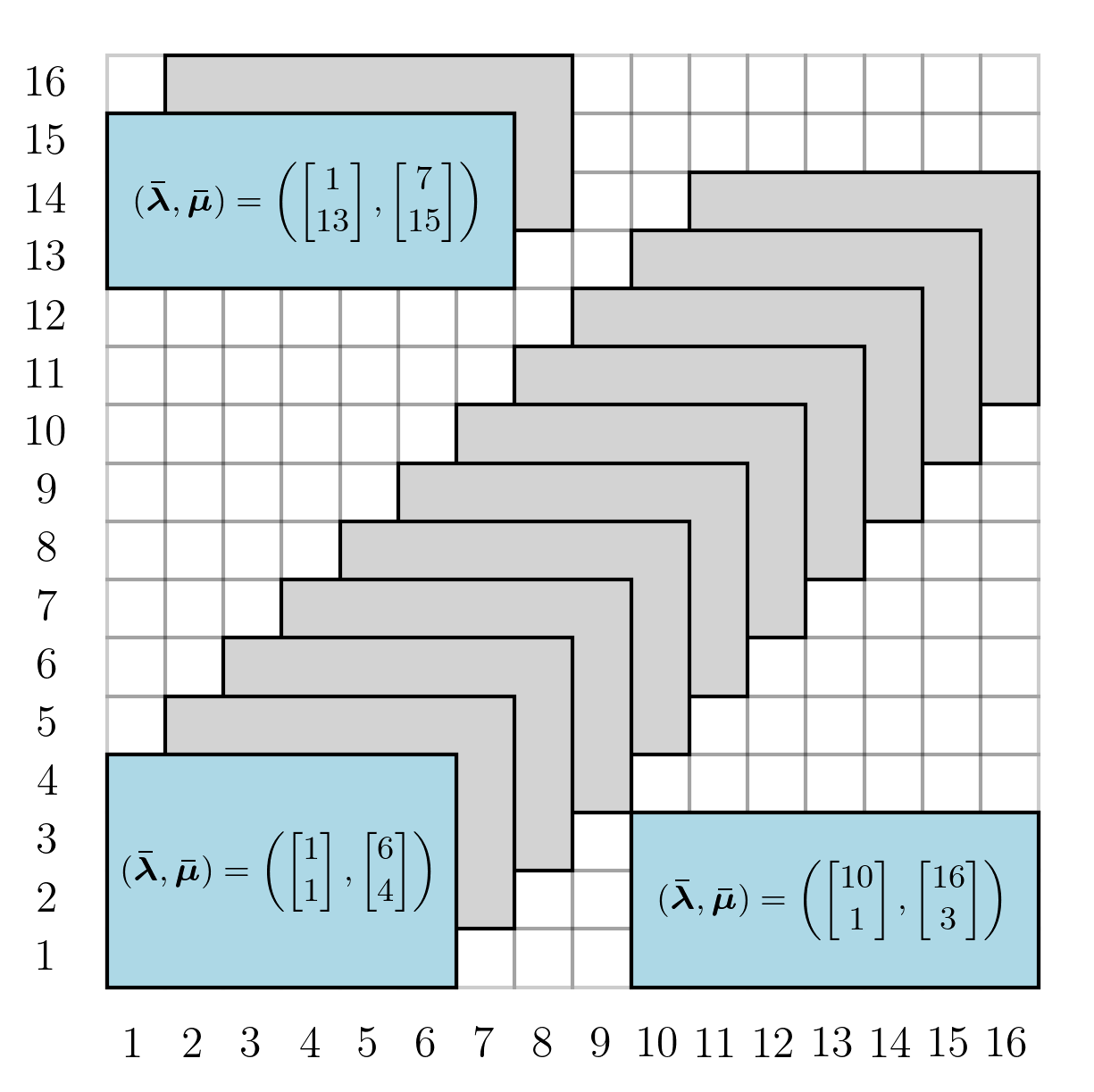}
\label{fig:lines_upright}}
\subfloat[]{
\includegraphics[width=0.4\textwidth]{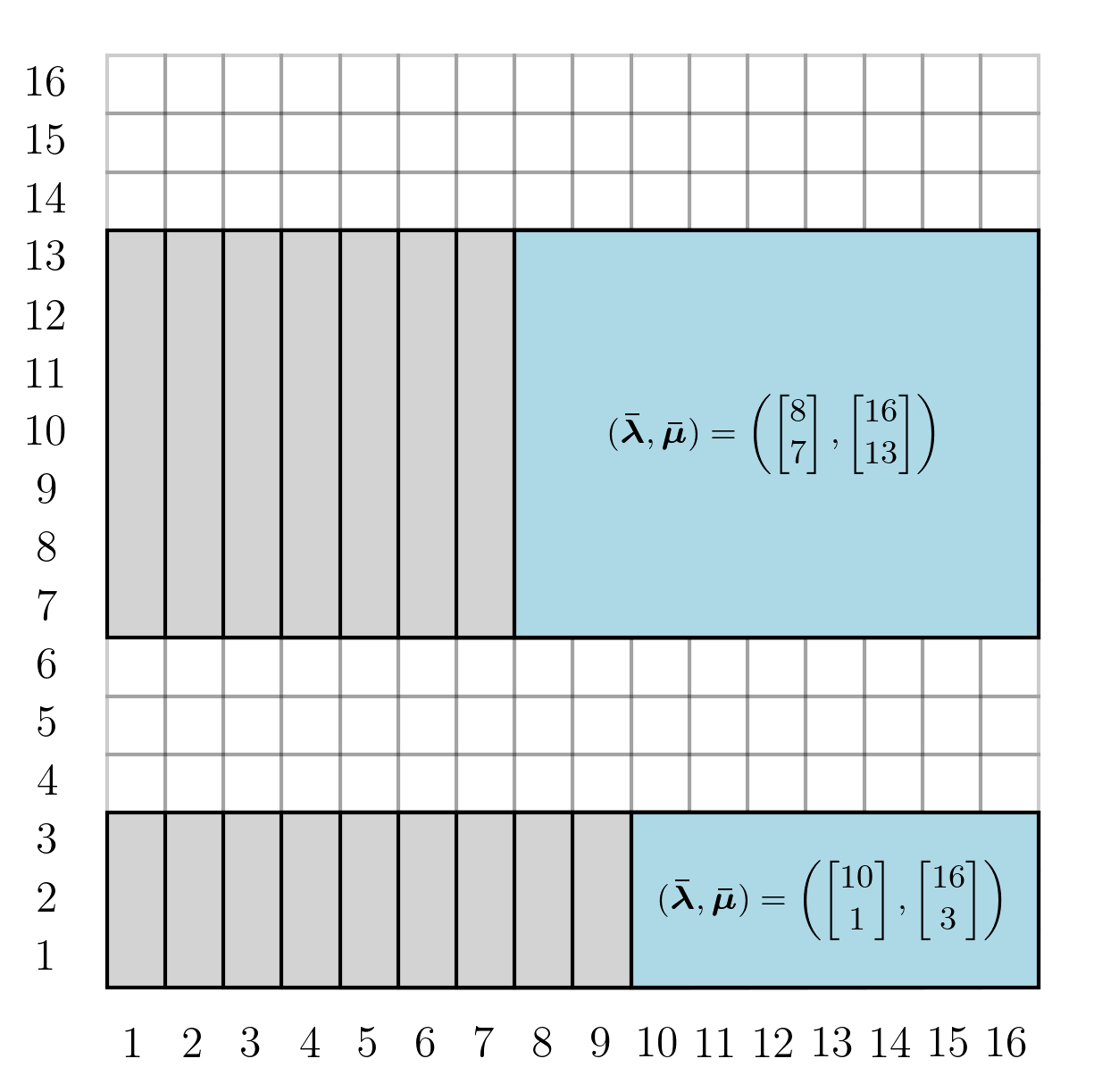}
\label{fig:lines_left}}
\end{minipage}
}
{Visualizations of $\bdelta$-anchors and $\bdelta$-lines\\[-3em]\label{fig:lines}} 
{Visualization of a subset of $\bdelta$-anchors and the $\bdelta$-lines emanating from those $\bdelta$-anchors for the case of $S = 16$. Specifically, \emph{(a)} shows three examples of $\bdelta$-anchors for the case of $\bdelta = (1,1)$, and \emph{(b)} shows two examples of $\bdelta$-anchors for the case of $\bdelta = (-1,0)$. In both \emph{(a)} and \emph{(b)}, the $\bdelta$-anchors are shown in blue. For each $\bdelta$-anchor $(\bar{\blambda},\bar{\bmu})$, the blue rectangle together with the corresponding grey rectangles constitute the $\bdelta$-line emanating from $(\bar{\blambda},\bar{\bmu})$. For example, in \emph{(a)}, the $\bdelta$-line emanating from the $\bdelta$-anchor $(\bar{\blambda},\bar{\bmu}) = ((1,1),(6,4))$ is equal to the set of buckets $\{((1,1) + \iota (1,1),(6,4) + \iota (1,1)): \iota \in \{0,\ldots, 9\}\} = \{((1,1),(6,4)),\ldots,((11,16),(11,14))\}$. 
} 
\end{figure}
\noindent A visualization of  $\bdelta$-anchors and $\bdelta$-lines is shown in Figure~\ref{fig:lines}. In the following two intermediary lemmas, we give a conservative upper bound on the number of $\bdelta$-anchors, and we 
 show that the $\bdelta$-lines emanating from the $\bdelta$-anchors form a cover of $\mathcal{B}^{\ge}_{S,M}$. 
\begin{lemma} \label{lem:num_anchors}
 $|\mathfrak{A}_{S,M}(\bdelta)| \le DS^{2D-1}$.
\end{lemma}
\begin{lemma} \label{lem:unique_line}
$\{\mathcal{L}_{S,M}(\bar{\blambda},\bar{\bmu},\bdelta): (\bar{\blambda},\bar{\bmu}) \in \mathfrak{A}_{S,M}(\bdelta) \}$ is a collection of sets whose union is equal to $\mathcal{B}^{\ge}_{S,M}$. 
\end{lemma}
It follows from Lemmas~\ref{lem:num_anchors} and \ref{lem:unique_line} that
\begin{align}
\left| \mathcal{A}^\pi_{S,M,\bdelta} \right| \le D S^{2D - 1} \max_{(\bar{\blambda},\bar{\bmu}) \in \mathfrak{A}_{S,M}(\bdelta)} \left|\mathcal{A}^\pi_{S,M,\bdelta} \cap \mathcal{L}_{S,M} \left( \bar{\blambda},\bar{\bmu},\bdelta \right) \right |. \label{line:bound_line}
\end{align}
Now consider any arbitrary $\bdelta$-anchor $(\bar{\blambda},\bar{\bmu}) \in \mathfrak{A}_{S,M}(\bdelta)$. To develop an upper bound on $|\mathcal{A}^\pi_{S,M,\bdelta} \cap \mathcal{L}_{S,M} \left( \bar{\blambda},\bar{\bmu},\bdelta \right) |$, we begin by decomposing $\mathcal{A}^\pi_{S,M,\bdelta}$ into $K$ subsets of buckets, 
\begin{align*}
\mathcal{A}^\pi_{S,M,\bdelta} 
&= \bigcup_{k=1}^K \underbrace{ \left \{ (\blambda,\bmu) \in \mathcal{B}_{S,M}^{\ge}: p_k = \min_{\bsigma \in \textsc{Face}(\blambda,\bmu,\bzero)} \min_{\bx \in \mathcal{H}_S(\bsigma)} \pi(\bx) \text{ and } p_k > \min \limits_{\bsigma \in \textsc{Face}(\blambda,\bmu,\bdelta)} \min \limits_{\bx \in \mathcal{H}_S(\bsigma)} \pi(\bx) \right \}}_{\triangleq \mathcal{A}^\pi_{S,M,\bdelta,k}},
\end{align*}
where the equality follows from the definition of $\mathcal{A}_{S,M,\bdelta}$ and from Assumption~\ref{ass:discrete}, which implies that $\pi$ is a function of the form $\pi: [0,1]^D \to \{p_1,\ldots,p_K\}$. Therefore, we observe that
\begin{align}
|\mathcal{A}_{S,M,\bdelta}^\pi \cap \mathcal{L}_{S,M} \left( \bar{\blambda},\bar{\bmu},\bdelta \right) |\le K \max_{k \in \{1,\ldots,K\}} \left| \mathcal{A}^\pi_{S,M,\bdelta,k} \cap \mathcal{L}_{S,M} \left( \bar{\blambda},\bar{\bmu},\bdelta \right) \right|. \label{line:bound_step4_1}
\end{align}
Now consider any arbitrary $k \in \{1,\ldots,K\}$. To develop an upper bound on 
$
|\mathcal{A}^\pi_{S,M,\bdelta,k} \cap \mathcal{L}_{S,M} ( \bar{\blambda},\bar{\bmu},\bdelta ) |,
$ 
our high-level strategy is to show that any two buckets $(\blambda,\bmu),(\blambda',\bmu')$ that are contained in $ \mathcal{A}^\pi_{S,M,\bdelta,k} \cap \mathcal{L}_{S,M}(\bar{\blambda},\bar{\bmu},\bdelta)$ must be sufficiently far apart from one another. We show this using an intermediary lemma denoted below by Lemma~\ref{lem:not_too_close}. The lemma shows that if two buckets on the $\bdelta$-line emanating from the $\bdelta$-anchor $(\bar{\blambda},\bar{\bmu})$ are sufficiently close to one another, then the $\bdelta$-face of the hyperrectangle corresponding to one bucket must be contained in the interior of the hyperrectangle corresponding to the other bucket. 
\begin{lemma} \label{lem:not_too_close}
Let $(\blambda,\bmu),(\blambda',\bmu') \in \mathcal{L}_{S,M}(\bar{\blambda},\bar{\bmu},\bdelta)$. If there exists an integer $\iota \in \{1,\ldots,M-1\}$ such that 
$(\blambda',\bmu') = (\blambda,\bmu) + \iota (\bdelta,\bdelta)$, then $\textsc{Face}(\blambda,\bmu,\bdelta) \subseteq \textsc{Face}(\blambda',\bmu',\bzero)$. 
\end{lemma}
To provide further intuition for Lemma~\ref{lem:not_too_close}, we recall that if two buckets satisfy $(\blambda,\bmu),(\blambda',\bmu') \in \mathcal{L}_{S,M}(\bar{\blambda},\bar{\bmu},\bdelta)$, there must exist nonnegative integers $\iota,\iota'$ such that $(\blambda,\bmu) = (\bar{\blambda},\bar{\bmu}) + \iota (\bdelta,\bdelta)$ and $(\blambda',\bmu') = (\bar{\blambda},\bar{\bmu}) + \iota' (\bdelta,\bdelta)$. If $ \iota' - \iota \in \{1,\ldots,M-1 \}$, then Lemma~\ref{lem:not_too_close} establishes that the $\bdelta$-face of the hyperrectangle corresponding to $(\blambda,\bmu)$ must be contained in the interior of the hyperrectangle corresponding to $(\blambda',\bmu')$. A visualization of Lemma~\ref{lem:not_too_close} is found in Figure~\ref{fig:main-idea}. 

\begin{figure}[t]
\centering
\FIGURE{
\begin{minipage}{\linewidth}
\centering
\includegraphics[width=0.86\textwidth]{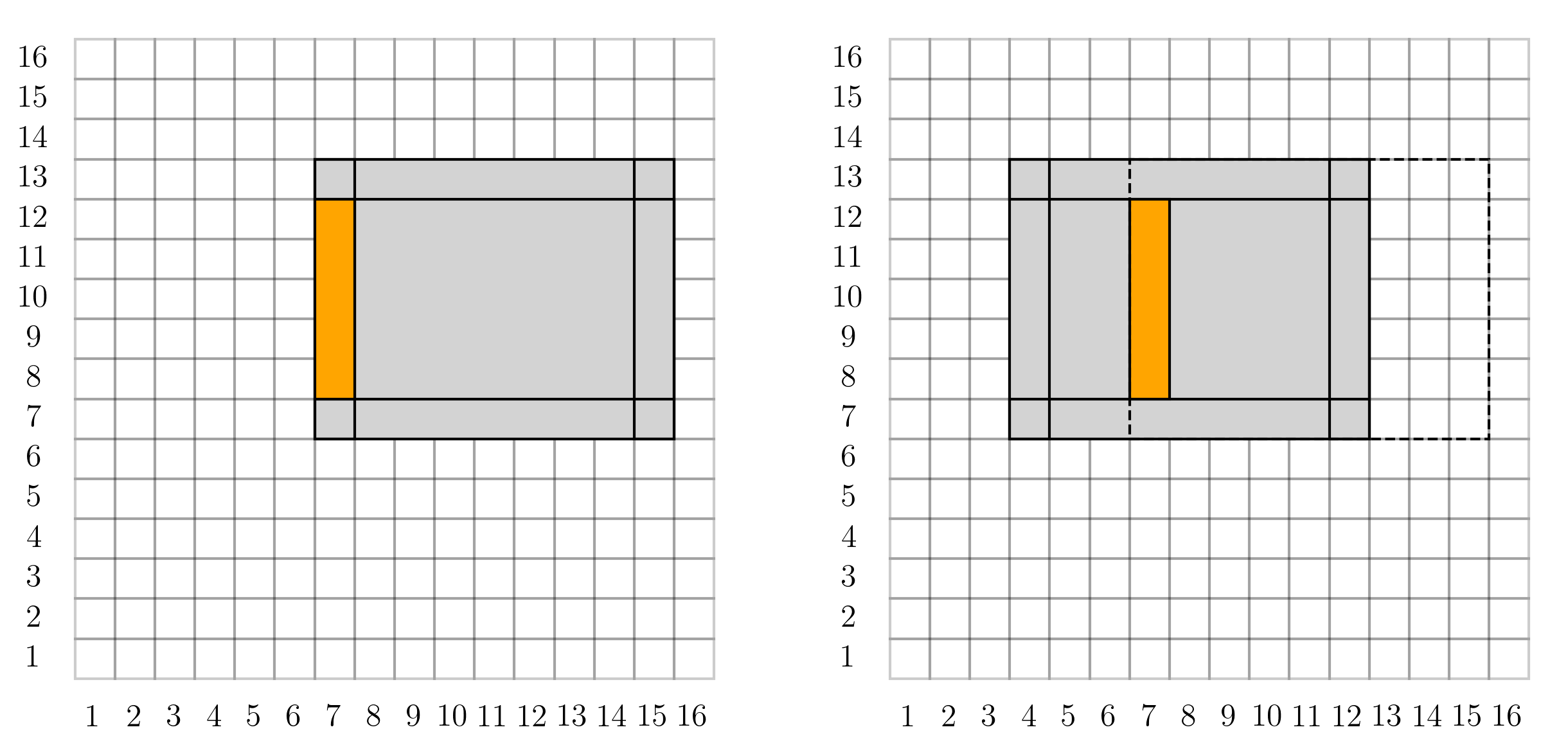}
\label{fig:main-idea}
\end{minipage}
}
{Visualization of Lemma~\ref{lem:not_too_close}} 
{Visualization of Lemma~\ref{lem:not_too_close} for the $\bdelta$-line emanating from the $\bdelta$-anchor $(\bar{\blambda},\bar{\bmu}) = ((8,7),(16,13))$ for $\bdelta = (-1,0)$ (see Figure~\ref{fig:lines}b). The gray rectangles in the left and right figures correspond to $\textsc{Rect}(\bar{\blambda} + \iota \bdelta, \bar{\bmu} + \iota \bdelta) $ for the case of $\iota = 1$ and $\iota = 4$, respectively. The orange rectangle in both figures corresponds to $\textsc{Face}(\bar{\blambda} + 1 \bdelta, \bar{\bmu} + 1 \bdelta, \bdelta)$. We observe from the right figure that $\textsc{Face}(\bar{\blambda} + 1 \bdelta, \bar{\bmu} + 1 \bdelta, \bdelta) \subseteq \textsc{Face}(\bar{\blambda} + 4 \bdelta, \bar{\bmu} + 4 \bdelta, \bzero)$. } 
\end{figure}

Equipped with Lemma~\ref{lem:not_too_close}, we proceed to establish our upper bound on $|\mathcal{A}^\pi_{S,M,\bdelta,k} \cap \mathcal{L}_{S,M} \left( \bar{\blambda},\bar{\bmu},\bdelta \right) |$. Indeed, let $\bar{\iota} \triangleq | \mathcal{L}_{S,M}(\bar{\blambda},\bar{\bmu},\bdelta)|$ denote the maximal integer such that $(\bar{\blambda},\bar{\bmu}) + \bar{\iota} (\bdelta,\bdelta) \in \mathcal{B}^{\ge}_{S,M}$. Note that it follows from the fact that $\bdelta \neq \bzero$ and from the fact that $(\bar{\blambda},\bar{\bmu}) \in \mathcal{B}^{\ge}_{S,M}$ that the integer $\bar{\iota}$ satisfies $\bar{\iota} \le S -1$. We observe that 
\begin{align}
\mathcal{A}^\pi_{S,M,\bdelta,k} \cap \mathcal{L}_{S,M} \left( \bar{\blambda},\bar{\bmu},\bdelta \right) 
= & \; \mathcal{A}^\pi_{S,M,\bdelta,k} \cap \{ (\bar{\blambda} + \iota \bdelta, \bar{\bmu} + \iota \bdelta): \iota \in \{0,\ldots,\bar{\iota} \} \} \notag \\[5pt]
= & \left\{ (\bar{\blambda} + \iota \bdelta, \bar{\bmu} + \iota \bdelta): \quad \begin{aligned} &\iota \in \{0,\ldots,\bar{\iota} \} \\
&p_k = \min \limits_{\bsigma \in \textsc{Face}(\bar{\blambda} + \iota \bdelta,\bar{\bmu} + \iota \bdelta ,\bzero)} \min \limits_{\bx \in \mathcal{H}_S(\bsigma)} \pi(\bx)\\
&p_k > \min \limits_{\bsigma \in \textsc{Face}(\bar{\blambda} + \iota \bdelta,\bar{\bmu} + \iota \bdelta ,\bdelta)} \min \limits_{\bx \in \mathcal{H}_S(\bsigma)} \pi(\bx) \end{aligned} \right \}, \label{line:rewrite_A_L}
\end{align}
where the first equality follows from the definition of $\mathcal{L}_{S,M}(\bar{\blambda},\bar{\bmu},\bdelta)$ and the second equality follows from the definition of $ \mathcal{A}^\pi_{S,M,\bdelta,k}$. Consider any arbitrary $\iota \in \{0,\ldots,\bar{\iota}\}$ that satisfies $(\bar{\blambda} + \iota \bdelta, \bar{\bmu} + \iota \bdelta) \in \mathcal{A}^\pi_{S,M,\bdelta,k} \cap \mathcal{L}_{S,M} \left( \bar{\blambda},\bar{\bmu},\bdelta \right) $ and any arbitrary $\iota' \in \{\iota + 1,\ldots,\min \{ \iota + M-1, \bar{\iota} \}\}$. It follows from line~\eqref{line:rewrite_A_L} that $p_k > \min_{\bsigma \in \textsc{Face}(\bar{\blambda} + \iota \bdelta,\bar{\bmu} + \iota \bdelta ,\bdelta)} \min_{\bx \in \mathcal{H}_S(\bsigma)} \pi(\bx) $, and it follows from Lemma~\ref{lem:not_too_close} that $\textsc{Face}(\blambda + \iota \bdelta,\bmu + \iota \bdelta,\bdelta) \subseteq \textsc{Face}(\blambda + \iota' \bdelta,\bmu + \iota' \bdelta,\bzero)$, which together imply that  $p_k > \min_{\bsigma \in \textsc{Face}(\bar{\blambda} + \iota' \bdelta,\bar{\bmu} + \iota' \bdelta ,\bzero)} \min_{\bx \in \mathcal{H}_S(\bsigma)} \pi(\bx)$. We thus conclude from line~\eqref{line:rewrite_A_L} that $(\bar{\blambda} + \iota' \bdelta, \bar{\bmu} + \iota' \bdelta) \notin \mathcal{A}^\pi_{S,M,\bdelta,k} \cap \mathcal{L}_{S,M} \left( \bar{\blambda},\bar{\bmu},\bdelta \right) $. More generally, we have shown that if two integers $\hat{\iota},\hat{\iota}' \in \{0,\ldots,\bar{\iota} \}$ satisfy $\hat{\iota} < \hat{\iota}'$ and $(\bar{\blambda} + \hat{\iota} \bdelta, \bar{\bmu} + \hat{\iota} \bdelta), (\bar{\blambda} + \hat{\iota}' \bdelta, \bar{\bmu} + \hat{\iota}' \bdelta) \in \mathcal{A}^\pi_{S,M,\bdelta,k} \cap \mathcal{L}_{S,M} \left( \bar{\blambda},\bar{\bmu},\bdelta \right)$, then it must be the case that $\hat{\iota}' - \hat{\iota} \ge M$. Therefore, we conclude that
\begin{align}
\left| \mathcal{A}^\pi_{S,M,\bdelta,k} \cap \mathcal{L}_{S,M} \left( \bar{\blambda},\bar{\bmu},\bdelta \right) \right| &\le \left \lceil \frac{\bar{\iota} + 1}{M} \right \rceil \le \left \lceil \frac{S}{M} \right \rceil, \label{line:bound_step4_2}
\end{align}
where the first inequality follows from our earlier reasoning that $\hat{\iota}' - \hat{\iota} \ge M$ for all $\hat{\iota},\hat{\iota}' \in \{0,\ldots,\bar{\iota} \}$ that satisfy $\hat{\iota} < \hat{\iota}'$ and $(\bar{\blambda} + \hat{\iota} \bdelta, \bar{\bmu} + \hat{\iota} \bdelta), (\bar{\blambda} + \hat{\iota}' \bdelta, \bar{\bmu} + \hat{\iota}' \bdelta) \in \mathcal{A}^\pi_{S,M,\bdelta,k} \cap \mathcal{L}_{S,M} \left( \bar{\blambda},\bar{\bmu},\bdelta \right)$, and the second inequality follows from our earlier observation that $\bar{\iota} \le S - 1$. 

Because the vector $\bdelta \in \{-1,0,1\}^D$ that satisfies $\bdelta \neq \bzero$, the $\bdelta$-anchor $(\bar{\blambda},\bar{\bmu}) \in \mathfrak{A}_{S,M}(\bdelta)$, and the integer $k \in \{1,\ldots,K\}$ were chosen arbitrarily, we conclude from lines~\eqref{line:bound_line}, \eqref{line:bound_step4_1}, and \eqref{line:bound_step4_2} that 
\begin{align*}
\left| \mathcal{A}^\pi_{S,M,\bdelta} \right| \le D S^{2D - 1} K \left \lceil \frac{S}{M} \right \rceil \le \frac{2 D S^{2D} K}{M},
\end{align*}
which completes the proof of Proposition~\ref{prop:bound_A_S_delta_omega}. 

\section{Numerical Experiment}\label{sec:numerics}
To illustrate the phase transition established by Theorem~\ref{thm:main}, we conduct numerical experiments that compare the non-strategic and strategic versions of a simple personalized pricing problem with two features, \textit{i.e.}, $D = 2$. The probability distribution for the buyer in our experiments is given by  $(\bbl,\bbu,V) = (\bbx - \bepsilon, \bbx + \bepsilon, V)$, where $\bbx$ is uniformly distributed over $[0.1,0.9]^2$, the radii $\bepsilon$ are chosen as either $\bepsilon = (0.09,0.09)$ or $\bepsilon = (0,0)$, and the buyer's valuation of the product is a deterministic function of $\bbl,\bbu$ of the form $V = L_1(1-L_1) + L_2(1-L_2) + U_1(1-U_1) + U_2(1-U_2).$ Our probability distribution is chosen such that the buyer has a higher (resp., lower) valuation if $\bbx$ is closer to the center (resp., the boundary) of $[0,1]^2$. The set of feasible prices for the seller is given by $\mathcal{P} = \{p_1,p_2 \}$ where $p_1 = 0.65$ and $p_2 = 0.83$. It can be readily verified that the above problem satisfies Assumptions~\ref{ass:iid}-\ref{ass:continuous} if  $\bepsilon = (0.09,0.09)$ and violates Assumption~\ref{ass:continuous} if $\bepsilon = (0,0)$. 

\begin{figure}[t]
\centering
\FIGURE{
\begin{minipage}{\linewidth}
\centering
\subfloat[]{
\includegraphics[width=0.45\textwidth]{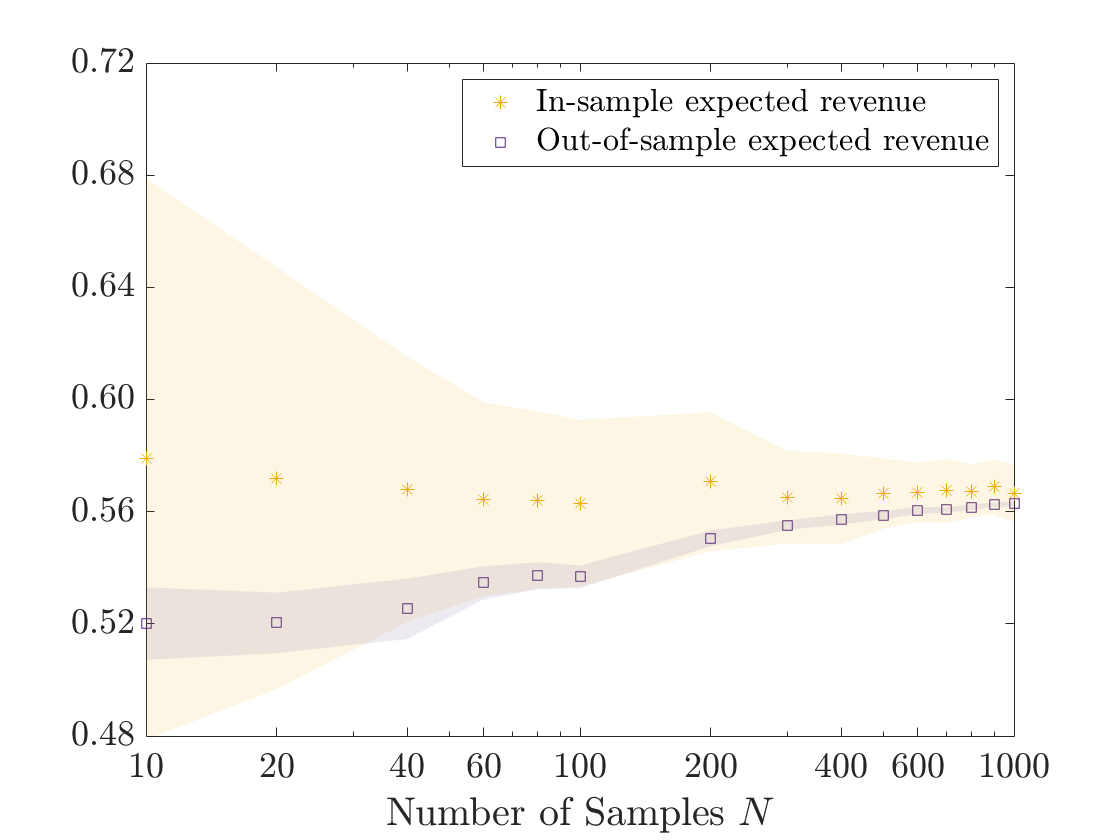}
\label{fig:str_performance}}
\subfloat[]{
\includegraphics[width=0.45\textwidth]{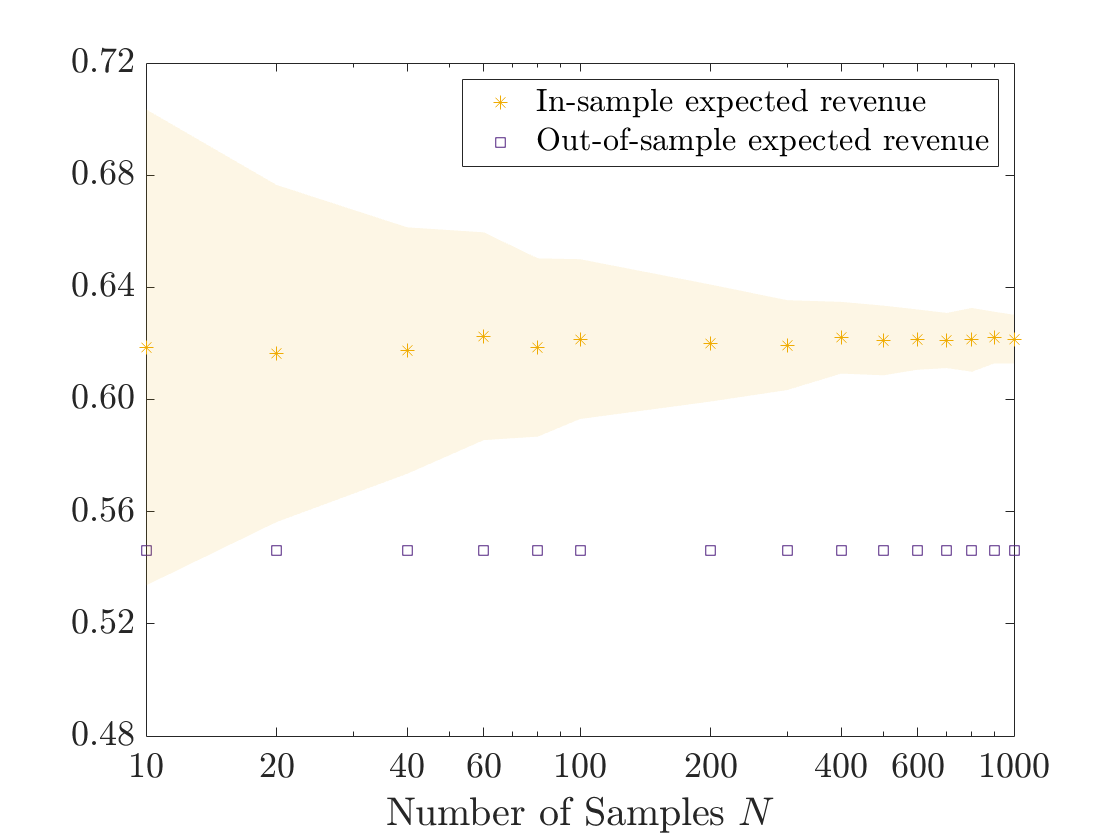}
\label{fig:non_performance}}
\end{minipage}
}
{Visualization of in-sample and out-of-sample performances.\\[-3em]\label{fig:numerical}}
{Each plot shows the (yellow) in-sample expected revenue $\widehat{J}_N(\hat{\pi}_N)$ and the (purple) out-of-sample expected revenue $J^*(\hat{\pi}_N)$, where $\hat{\pi}_N$ is an optimal solution for \eqref{prob:saa} with (\emph{a}) $\bepsilon = (0.09,0.09)$ and (\emph{b}) $\bepsilon = (0,0)$. The solid lines describe the mean statistics over $50$ random instances, whereas the shaded regions cover mean $\pm$ standard deviation. The experiment details for solving \eqref{prob:saa} can be found in Appendix~\ref{sec:MILP}.}
\end{figure}

In Figure~\ref{fig:numerical}, we visualize the phase transition between the strategic and non-strategic versions of~\eqref{prob:saa}. Specifically, let $\widehat{\pi}_N$ denote an optimal solution for \eqref{prob:saa}. 
In Figure~\ref{fig:numerical}a, we observe for the case of $\bepsilon = (0.09,0.09)$ that the in-sample expected revenue $\widehat{\nu}_N \equiv \widehat{J}_N(\widehat{\pi}_N)$ and the out-of-sample expected revenue $J^*(\widehat{\pi}_N)$ of the pricing policy $\widehat{\pi}_N$ converge to one another as the number of samples $N$ increases. That is, Figure~\ref{fig:numerical}a illustrates  Theorem~\ref{thm:main} by showing that $\widehat{\nu}_N$ converges to $\nu^*$. In contrast, Figure~\ref{fig:numerical}b shows for the case of $\bepsilon = (0,0)$ that the in-sample objective value of~\eqref{prob:saa} always overfits, even as the number of samples $N$ increases.

\section{Omitted Proofs from \S\ref{sec:proof:thm:main}} \label{sec:intermediary_lemmas}
\subsection{Relaxation of Assumption~\ref{ass:discrete}} \label{sec:proof:remove_ass}
To show that Assumption~\ref{ass:discrete} is unnecessary for the proof of Theorem~\ref{thm:main}, we prove the following:
\begin{proposition} \label{prop:finiteprices}
 Let Assumptions~\ref{ass:iid},  \ref{ass:bounded} hold. For each $K > 0$, there exist $p_1,\ldots,p_K \in \mathcal{P}$ such that\looseness=-1
 \begin{align*}
 0 \le \sup \limits_{\pi: [0,1]^D \to \mathcal{P}} \frac{1}{N} \sum \limits_{i=1}^N R^\pi\left(\bbl^i,\bbu^i,V^i \right) - \sup \limits_{\pi: [0,1]^D \to \{p_1,\ldots,p_K \}} \frac{1}{N} \sum \limits_{i=1}^N R^\pi\left(\bbl^i,\bbu^i,V^i \right) \le \frac{1}{K}
\end{align*}
for every integer $N >0$, almost surely, and
\begin{align*}
 0 \le \sup \limits_{\pi: [0,1]^D \to \mathcal{P}}\Exp \left[ R^\pi\left(\bbl,\bbu,V \right) \right] - \sup \limits_{\pi: [0,1]^D \to \{p_1,\ldots,p_K \}} \Exp \left[ R^\pi\left(\bbl,\bbu,V\right) \right] \le \frac{1}{K}.
\end{align*}
\end{proposition}
The above proposition implies that Assumption~\ref{ass:discrete} is unnecessary for the proof of Theorem~\ref{thm:main} because the proposition shows, for every arbitrary $K > 0$, that there exist prices $p_1,\ldots,p_K \in \mathcal{P}$ such that \eqref{prob:opt} and \eqref{prob:saa} can be approximated within a $\frac{1}{K}$-additive error by restricting to pricing policies of the form $\pi: [0,1]^D \to \{p_1,\ldots,p_K\}$. 
To prove Proposition~\ref{prop:finiteprices}, we use the following lemmas. 
\begin{lemma} \label{lem:finiteprices}
Let Assumption~\ref{ass:bounded} hold. For every integer $K > 0$, there exist prices $p_1,\ldots,p_K \in \mathcal{P}$ such that for every $p \in \mathcal{P}$, there exists $k \in \{1,\ldots,K\}$ such that 
$0 \le p - p_k \le \frac{1}{K}$.
\end{lemma}
\begin{proof}{Proof of Lemma~\ref{lem:finiteprices}.}
Consider any integer $K > 0$. The Weierstrass extreme value theorem, in combination with the fact that the set $\mathcal{P} \neq \emptyset$ is closed and bounded (Assumption~\ref{ass:bounded}), implies that any continuous real-valued function over $\mathcal{P}$ must attain a maximum and minimum. Hence, we have that the quantities $\ubar{p},\bar{p},p_1,\ldots,p_K$ defined below exist and are elements of $\mathcal{P}$:
\begin{align*}
 \ubar{p} \triangleq \inf_{p \in \mathcal{P}} p;\quad \bar{p} \triangleq \sup_{p \in \mathcal{P}} p; \quad 
 p_k \triangleq \inf \left \{ p \in \mathcal{P}: p \ge \ubar{p} + ( \bar{p} - \ubar{p}) \frac{k-1}{{K}} \right \} \quad \forall k \in \{1,\ldots,K\}. 
\end{align*}
Now consider any arbitrary $p \in \mathcal{P}$. It follows from the fact that $p \in \mathcal{P} \subseteq [\ubar{p},\bar{p}]$ that there must exist some $k \in \{1,\ldots,K\}$ such that $
 \ubar{p} + ( \bar{p} - \ubar{p}) \frac{k-1}{{K}} \le p \le \ubar{p} + ( \bar{p} - \ubar{p}) \frac{k}{{K}}.$ 
For that $k$, we observe that $p_k = \inf \left \{ p' \in \mathcal{P}: \ubar{p} + ( \bar{p} - \ubar{p}) \frac{k-1}{K} \le p' \right \}\le p,$  where the inequality follows from the fact that $ \ubar{p} + ( \bar{p} - \ubar{p}) \frac{k-1}{{K}} \le p$ and from the fact that $p \in \mathcal{P}$. We have thus shown that $
 p_k \le p \le \ubar{p} + ( \bar{p} - \ubar{p}) \frac{k}{K}$, 
which implies that 
\begin{align}
0 \le p - p_k &\le \ubar{p} + ( \bar{p} - \ubar{p}) \frac{k}{{K}} - p_k \notag \\
&\le \ubar{p} + ( \bar{p} - \ubar{p}) \frac{k}{K} - \left( \ubar{p} + ( \bar{p} - \ubar{p}) \frac{k - 1}{K} \right) \label{line:substitute_p_ell}\\
&= \left(\bar{p} - \ubar{p} \right) \frac{1}{K} \notag \\
&\le \frac{1}{K}, \label{line:pbar_bounded} 
\end{align}
where line~\eqref{line:substitute_p_ell} follows from the fact that $p_k \ge \ubar{p} + (\bar{p} - \ubar{p}) \frac{k-1}{K}$ and line~\eqref{line:pbar_bounded} follows from the fact that $\mathcal{P} \subseteq [0,1]$. Since $p \in \mathcal{P}$ was chosen arbitrarily, our proof of Lemma~\ref{lem:finiteprices} is complete. 
\hfill \halmos 
\end{proof}

\begin{lemma} \label{lem:finiteprices2}
Let Assumption \ref{ass:bounded} hold. For every integer $K > 0$, there exist prices $p_1,\ldots,p_K \in \mathcal{P}$ such that for every pricing policy $\pi: [0,1]^D \to \mathcal{P}$, there exists a pricing policy $\pi': [0,1]^D \to \{p_1,\ldots,p_K\}$  such that for all $(\bell,\bu,v) \in \mathcal{F} \times [0,1]$ we have $R^{\pi'}(\bell,\bu,v) \ge R^{\pi}(\bell,\bu,v) - \frac{1}{K}$.
\end{lemma}
\begin{proof}{Proof of Lemma~\ref{lem:finiteprices2}.}
Consider any integer $K > 0$, and let $p_1,\ldots,p_K \in \mathcal{P}$ be the prices obtained from Lemma~\ref{lem:finiteprices}. Consider any pricing policy $\pi: [0,1]^D \to \mathcal{P}$. It follows from Lemma~\ref{lem:finiteprices} and from the fact that $\pi: [0,1]^D \to \mathcal{P}$ that there exists a pricing  policy $\pi': [0,1]^D \to \{p_1,\ldots,p_K\}$  satisfying $ 0 \le \pi(\bx) - \pi'(\bx) \le \frac{1}{K}$ for all $\bx \in [0,1]^D$. Therefore, for all $(\bell,\bu,v) \in \mathcal{F} \times [0,1]$ we have 
\begin{align*}
R^\pi\left(\bell,\bu,v \right) &= \mathbb{I} \left \{ \inf_{\bx \in [\bell,\bu]} \pi (\bx ) \le v \right \} \cdot \inf_{\bx \in [\bell,\bu]} \pi (\bx )
\le \mathbb{I} \left \{ \inf_{\bx \in [\bell,\bu]} \pi' (\bx ) \le v \right \} \cdot \inf_{\bx \in [\bell,\bu]} \pi (\bx )\\
&\le \mathbb{I} \left \{ \inf_{\bx \in [\bell,\bu ]} \pi' (\bx ) \le v \right \} \cdot \inf_{\bx \in [\bell,\bu]} \left\{ \pi' (\bx ) + \frac{1}{K} \right \} 
\le \frac{1}{K} + \mathbb{I} \left \{ \inf_{\bx \in [\bell,\bu]} \pi' (\bx ) \le v \right \} \cdot \inf_{\bx \in [\bell,\bu] } \pi' (\bx ) \\
&= \frac{1}{K} + R^{\pi'}\left(\bell,\bu,v \right). 
\end{align*}
The first equality is the definition of $R^{\pi}(\bell,\bu,v)$. The first inequality holds because $0 \le \pi(\bx) - \pi'(\bx)$. The second inequality holds because $\pi(\bx) - \pi'(\bx) \le \frac{1}{K}$. The third inequality follows from algebra, and the second equality follows from the definition of $R^{\pi'}(\bell,\bu,v).$ 
\hfill \halmos 
\end{proof}

We conclude \S\ref{sec:proof:remove_ass} by using Lemma~\ref{lem:finiteprices2} to prove Proposition~\ref{prop:finiteprices}. 
\begin{proof}{Proof of Proposition~\ref{prop:finiteprices}.}
Consider any arbitrary integer $K > 0$, and let $p_1,\ldots,p_K \in \mathcal{P}$ be the prices obtained from Lemma~\ref{lem:finiteprices2}. For each arbitrary integer $N > 0$, it follows from algebra that 
\begin{align*}
0 \le \sup \limits_{\pi: [0,1]^D \to \mathcal{P}} \frac{1}{N} \sum \limits_{i=1}^N R^\pi\left(\bbl^i,\bbu^i,V^i \right) - \sup \limits_{\pi: [0,1]^D \to \{p_1,\ldots,p_K \}} \frac{1}{N} \sum \limits_{i=1}^N R^\pi\left(\bbl^i,\bbu^i,V^i \right).
\end{align*}
Moreover, choose any arbitrary $\eta > 0$, and let $\widehat{\pi}_{N,\eta}: [0,1]^D \to \mathcal{P}$ denote a pricing policy that satisfies
\begin{align*}
\frac{1}{N} \sum \limits_{i=1}^N R^{\widehat{\pi}_{N,\eta}}\left(\bbl^i,\bbu^i,V^i \right) \geq \sup \limits_{\pi: [0,1]^D \to \mathcal{P}} \frac{1}{N} \sum \limits_{i=1}^N R^\pi\left(\bbl^i,\bbu^i,V^i \right) - \eta.
\end{align*}
It follows from Lemma~\ref{lem:finiteprices2} that there exists a pricing policy $\widehat{\pi}_{N,\eta}': [0,1]^D \to \{p_1,\ldots,p_K\}$ that satisfies 
$R^{\widehat{\pi}'_{N,\eta}}(\bell,\bu,v) \ge R^{\widehat{\pi}_{N,\eta}}(\bell,\bu,v) - \frac{1}{K}$ for all $(\bell,\bu,v) \in \mathcal{F} \times [0,1]$. Therefore,
\begin{align*}
\sup \limits_{\pi: [0,1]^D \to \mathcal{P}} \frac{1}{N} \sum \limits_{i=1}^N R^\pi\left(\bbl^i,\bbu^i,V^i \right) &\le \eta + \frac{1}{N} \sum \limits_{i=1}^N R^{\widehat{\pi}_{N,\eta}}\left(\bbl^i,\bbu^i,V^i \right) \\
&\le \eta + \frac{1}{K} + \frac{1}{N} \sum \limits_{i=1}^N R^{\widehat{\pi}'_{N,\eta}}\left(\bbl^i,\bbu^i,V^i \right) \text{ almost surely}\\
&\le \eta + \frac{1}{K} + \sup_{\pi: [0,1]^D \to \{p_1,\ldots,p_K \}} \frac{1}{N} \sum \limits_{i=1}^N R^{\pi}\left(\bbl^i,\bbu^i,V^i \right).
\end{align*}
The first inequality follows from the definition of $\widehat{\pi}_{N,\eta}$. The second inequality follows from the fact that $(\bbl^1,\bbu^1,V^1),\ldots,(\bbl^N,\bbu^N,V^N) \in \mathcal{F} \times [0,1]$ almost surely (Assumption~\ref{ass:iid}). The third inequality follows from the fact that $\widehat{\pi}'_{N,\eta}$ is a feasible but possibly suboptimal solution for the optimization problem $\sup_{\pi: [0,1]^D \to \{p_1,\ldots,p_K \}} \frac{1}{N} \sum \limits_{i=1}^N R^{\pi}\left(\bbl^i,\bbu^i,V^i \right).$ Since $N > 0$ and $\eta > 0$ were chosen arbitrarily, we conclude that the following inequality holds almost surely for every integer $N > 0$:
\begin{align*}
\sup \limits_{\pi: [0,1]^D \to \mathcal{P}} \frac{1}{N} \sum \limits_{i=1}^N R^\pi\left(\bbl^i,\bbu^i,V^i \right) - \sup \limits_{\pi: [0,1]^D \to \{p_1,\ldots,p_K \}} \frac{1}{N} \sum \limits_{i=1}^N R^\pi\left(\bbl^i,\bbu^i,V^i \right) \le \frac{1}{K}. \end{align*}
An identical reasoning is used to show the bound for the expectations. 
\hfill \halmos 
\end{proof}

\subsection{Proofs of Intermediary Lemmas} \label{sec:proof:lemmas}
\begin{proof}{Proof of Lemma~\ref{lem:characterization_customers}.}
Consider any $(\blambda,\bmu) \in \mathcal{A}^\pi_{S}$. We use the following two intermediary claims.
\begin{claim} \label{claim:characterization_customers:interior}
If $(\bell,\bu,v) \in \mathcal{C}_S(\blambda,\bmu)$, then $\cup_{\bsigma \in \textsc{Face}(\blambda,\bmu,\bzero)} \mathcal{H}_S(\bsigma) \subseteq [\bell,\bu]$.
\end{claim}
\begin{proof}{Proof of Claim~\ref{claim:characterization_customers:interior}.}
It follows from the definition of $\mathcal{C}_S(\blambda,\bmu)$ that the following holds for all buyer realizations $(\bell,\bu,v) \in \mathcal{C}_S(\blambda,\bmu)$ and every $d \in \{1,\ldots,D\}$:
\begin{align}
\ell_d \in \mathcal{H}_{S,d}(\lambda_d) \subseteq \left[ \frac{\lambda_d - 1}{S}, \frac{\lambda_d}{S} \right] \text{ and } u_d \in \mathcal{H}_{S,d}(\mu_d) \subseteq \left[ \frac{\mu_d - 1}{S}, \frac{\mu_d}{S} \right]. \label{line:characterization:inclusion} 
\end{align}
Thus, for every $(\bell,\bu,v) \in \mathcal{C}_S(\blambda,\bmu)$, $\bsigma \in \textsc{Face}(\blambda,\bmu,\bzero)$, and $d \in \{1,\ldots,D\}$, we have 
\begin{align}
\ell_d \le \frac{\lambda_d}{S} \le \frac{\sigma_d - 1}{S} 
\text{ and } 
u_d \ge \frac{\mu_d - 1}{S} \ge \frac{\sigma_d}{S}, \label{line:bound_ell_and_u}
\end{align}
where the first inequalities for $\ell_d$ and $u_d$ follow from \eqref{line:characterization:inclusion}, and the second inequalities follow from the fact that $\bsigma \in \textsc{Face}(\blambda,\bmu,\bzero)$ (which implies that $\sigma_d \in \{\lambda_d + 1,\ldots,\mu_d - 1 \}$). We conclude from \eqref{line:bound_ell_and_u} that for every buyer realization $(\bell,\bu,v) \in \mathcal{C}_S(\blambda,\bmu)$, $\bsigma \in \textsc{Face}(\blambda,\bmu,\bzero)$, and $d \in \{1,\ldots,D\}$, it must be the case that $\mathcal{H}_{S,d}(\sigma_d) \subseteq \left[ \frac{\sigma_d - 1}{S}, \frac{\sigma_d}{S}\right] \subseteq \left[ \ell_d,u_d \right].$ 
We have therefore shown for every buyer realization $(\bell,\bu,v) \in \mathcal{C}_S(\blambda,\bmu)$ and $\bsigma \in \textsc{Face}(\blambda,\bmu,\bzero)$ that $\mathcal{H}_{S}(\bsigma) \subseteq \left[ \bell,\bu \right].$
Since this reasoning holds for every $\bsigma \in \textsc{Face}(\blambda,\bmu,\bzero),$ our proof of Claim~\ref{claim:characterization_customers:interior} is complete. 
\hfill \halmos 
\end{proof}

\begin{claim} \label{claim:characterization_customers:boundary}
If $(\bell,\bu,v) \in \mathcal{C}_S(\blambda,\bmu)$, then $[\bell,\bu] \subseteq \cup_{\bsigma \in \textsc{Rect}(\blambda,\bmu)} \mathcal{H}_S(\bsigma) $.
\end{claim}
\begin{proof}{Proof of Claim~\ref{claim:characterization_customers:boundary}.}
It follows from the definition of $\mathcal{C}_S(\blambda,\bmu)$ that the following holds for all buyer realizations $(\bell,\bu,v) \in \mathcal{C}_S(\blambda,\bmu)$ and every $d \in \{1,\ldots,D\}$:
\begin{align}
\frac{\lambda_d - 1}{S} \le \ell_d \text{ and } \begin{cases} u_d < \frac{\mu_d}{S} &\text{if } \mu_d < S \\
u_d \le \frac{\mu_d}{S} &\text{if } \mu_d = S.
\end{cases} \label{line:characterization:inclusion2} 
\end{align}
Thus, it follows from the fact that $\ell_d \le u_d$ that  
\begin{align*}
\left[ \ell_d, u_d \right] &\subseteq \begin{cases} \left[\frac{\lambda_d-1}{S}, \frac{\mu_d}{S}\right) &\text{if } \mu_d < S \\
\left[\frac{\lambda_d-1}{S}, \frac{\mu_d}{S}\right] &\text{if } \mu_d = S
\end{cases} \\
&= \mathcal{H}_{S,d}(\lambda_d) \cup \cdots \cup \mathcal{H}_{S,d}(\mu_d).
\end{align*}
Hence, we conclude for every buyer realization $(\bell,\bu,v) \in \mathcal{C}_S(\blambda,\bmu)$ that
\begin{align*}
[\bell,\bu] &= \bigtimes_{d=1}^D \left[\ell_d,u_d \right] \subseteq \bigtimes_{d=1}^D \bigcup_{\sigma_d \in \left\{\lambda_d,\ldots,\mu_d \right \}} \mathcal{H}_{S,d}\left(\sigma_d \right) = \bigcup_{\bsigma \in \textsc{Rect}(\bell,\bu)} \mathcal{H}_S(\bsigma),
\end{align*}
which completes the proof of Claim~\ref{claim:characterization_customers:boundary}. 
\hfill \halmos 
\end{proof}

Equipped with Claims~\ref{claim:characterization_customers:interior} and \ref{claim:characterization_customers:boundary}, we proceed to complete the proof of Lemma~\ref{lem:characterization_customers}. Indeed, it follows from the fact that $(\blambda,\bmu) \in \mathcal{A}^\pi_{S}$ that there exist a buyer realization $(\bell,\bu,v) \in \mathcal{C}_S(\blambda,\bmu)$ and $k \in \{1,\ldots,K\}$  satisfying $R^{T_S \circ \pi}(\bell,\bu,v) < R^\pi(\bell,\bu,v) = p_k$. Consider any such buyer realization and $k$. We observe that
\begin{align}
\min_{\bsigma \in \textsc{Face}(\blambda,\bmu,\bzero)} \min_{\bx \in \mathcal{H}_S(\bsigma)} \pi(\bx) = \min_{\bx \in \cup_{\bsigma \in \textsc{Face}(\blambda,\bmu,\bzero)} \mathcal{H}_S(\bsigma)} \pi(\bx) \ge \min_{\bx \in [\bell,\bu]}\pi(\bx) = R^\pi(\bell,\bu,v) = p_k. \label{line:characterization_customers:oneside}
\end{align}
Indeed, the first equality follows from algebra. The inequality follows from Claim~\ref{claim:characterization_customers:interior}. The second equality follows from Assumption~\ref{ass:discrete} (which implies that $p_1,\ldots,p_K > 0$) and from the fact that $R^\pi(\bell,\bu,v) = p_k$, which together imply that $\mathbb{I} \left \{ \min_{\bx \in [\bell,\bu]} \pi(\bx) \le v \right \} =1$. The third equality follows from our choice of the buyer realization $(\bell,\bu,v) \in \mathcal{C}_S(\blambda,\bmu)$. Moreover, we observe that
\begin{align}
p_k &> R^{T_S \circ \pi}(\bell,\bu,v) 
= \mathbb{I} \left \{\min_{\bx \in [\bell,\bu]}(T_S \circ \pi)(\bx) \le v \right \} \cdot \min_{\bx \in [\bell,\bu]}(T_S \circ \pi)(\bx) \notag\\
&\ge \mathbb{I} \left \{\min_{\bx \in [\bell,\bu]}\pi(\bx) \le v \right \} \cdot \min_{\bx \in [\bell,\bu]}(T_S \circ \pi) (\bx) 
= \min_{\bx \in [\bell,\bu]}(T_S \circ \pi)(\bx)\notag \\
&\ge \min_{\bx \in \cup_{\bsigma \in \textsc{Rect}(\blambda,\bmu)} \mathcal{H}_S(\bsigma)} (T_S \circ \pi)(\bx)
= \min_{\bsigma \in \textsc{Rect}(\blambda,\bmu)} \min_{\bx \in \mathcal{H}_S(\bsigma)} (T_S \circ \pi) (\bx) \notag \\
&= \min_{\bsigma \in \textsc{Rect}(\blambda,\bmu)} \min_{\bx \in \mathcal{H}_S(\bsigma)} \min_{\bx' \in \mathcal{H}_S(\bsigma_S(\bx))} \pi(\bx') 
= \min_{\bsigma \in \textsc{Rect}(\blambda,\bmu)} \min_{\bx \in \mathcal{H}_S(\bsigma)} \pi(\bx) \notag \\
&= \min_{\bdelta \in \{-1,0,1\}^D} \min_{\bsigma \in \textsc{Face}(\blambda,\bmu,\bdelta)} \min_{\bx \in \mathcal{H}_S(\bsigma)} \pi(\bx) 
= 
\min_{\bdelta \in \{-1,0,1\}^D: \bdelta \neq \bzero} \min_{\bsigma \in \textsc{Face}(\blambda,\bmu,\bdelta)} \min_{\bx \in \mathcal{H}_S(\bsigma)} \pi(\bx).\label{line:characterization_customers:otherside}
\end{align}
Indeed, the first inequality follows from our choice of the buyer realization $(\bell,\bu,v) \in \mathcal{C}_S(\blambda,\bmu)$. The first equality is the definition of $R^{T_S \circ \pi}(\bell,\bu,v)$. The second inequality follows from the definition of $T_S \circ \pi$, which implies that $(T_S \circ \pi)(\bx) \le \pi(\bx)$ for all $\bx \in [0,1]^D$. The second equality follows from the fact that $R^\pi(\bell,\bu,v) > 0$, which implies that $ \mathbb{I} \left \{\min_{\bx \in [\bell,\bu]}\pi(\bx) \le v \right \} = 1$. The third inequality follows from Claim~\ref{claim:characterization_customers:boundary}. The third equality follows from algebra. The fourth equality is the definition of $T_S \circ \pi$. The fifth equality follows from the fact that $\bsigma_S(\bx) = \bsigma$ for all $\bx \in \mathcal{H}_S(\bsigma)$. The sixth equality follows from the fact that $\textsc{Rect}(\blambda,\bmu) = \cup_{\bsigma \in \{-1,0,1\}^D} \textsc{Face}(\blambda,\bmu,\bdelta)$. The seventh equality follows from the fact that $ \min_{\bsigma \in \textsc{Face}(\blambda,\bmu,\bzero)} \min_{\bx \in \mathcal{H}_S(\bsigma)} \pi(\bx) \ge p_k$, which is established in line~\eqref{line:characterization_customers:oneside}. Combining lines~\eqref{line:characterization_customers:oneside} and \eqref{line:characterization_customers:otherside}, our proof of Lemma~\ref{lem:characterization_customers} is complete. 
\hfill \halmos 
\end{proof}

\begin{proof}{Proof of Lemma~\ref{lem:proportion_bucket}.}
For notational convenience, let $\bbz \triangleq (\bbl,\bbu) \in \mathcal{F}$. For each Borel set $\mathcal{Z} \in \mathscr{B}(\R^{2D})$, recall that $\int_{\mathcal{Z}} d \bz$ denotes the volume of $\mathcal{Z}$ with respect to the Lebesgue measure. Our proof of Lemma~\ref{lem:proportion_bucket} will use the following intermediary claim. 

\begin{claim} \label{claim:second_inequality:probability:intermediary}
$\lim \limits_{\rho \downarrow 0}\sup \limits_{\mathcal{Z} \in \mathscr{B}(\R^{2D}): \int_{\mathcal{Z}} d \bz \le \rho} \Prb(\bbz \in \mathcal{Z}) = 0. $
\end{claim}

\begin{proof}{Proof of Claim~\ref{claim:second_inequality:probability:intermediary}.}
We observe that the function $h(\rho) \triangleq \sup_{\mathcal{Z} \in \mathscr{B}(\R^{2D}): \int_{\mathcal{Z}} d \bz \le \rho} \Prb(\bbz \in \mathcal{Z}) $ is bounded and monotonically increasing for $\rho > 0$. Therefore, it follows from the fact that $h(\rho) \in [0,1]$ that the limit $\lim_{\rho \downarrow 0}h(\rho)$ exists. Hence, 
\begin{align}\label{line:second_inequality:probability:1}
\lim \limits_{\rho \downarrow 0}\sup \limits_{\mathcal{Z} \in \mathscr{B}(\R^{2D}): \int_{\mathcal{Z}} d \bz \le \rho} \Prb(\bbz \in \mathcal{Z}) &= \lim_{n \to \infty} \sup \limits_{\mathcal{Z} \in \mathscr{B}(\R^{2D}): \int_{\mathcal{Z}} d \bz \le \frac{1}{n^2}} \Prb(\bbz \in \mathcal{Z}).
\end{align}
Choose any arbitrary constant $\eta > 0$. For each integer $n \in \N$, let $\mathcal{Z}_n^\eta \in \mathscr{B}(\R^{2D})$ denote a Borel set that satisfies both the inequality $\int_{\mathcal{Z}_n^\eta} d \bz \le \frac{1}{n^2}$ and the inequality 
$\sup_{\mathcal{Z} \in \mathscr{B}(\R^{2D}): \int_{\mathcal{Z}} d \bz \le \frac{1}{n^2}} \Prb(\bbz \in \mathcal{Z}) \le \Prb\left(\bbz \in \mathcal{Z}^\eta_n \right) + \eta.$ 
Therefore, we observe that
\begin{align}
\lim_{n \to \infty} \sup \limits_{\mathcal{Z} \in \mathscr{B}(\R^{2D}): \int_{\mathcal{Z}} d \bz \le \frac{1}{n^2}} \Prb(\bbz \in \mathcal{Z})\le \limsup_{n \to \infty} \Prb \left(\bbz \in \mathcal{Z}^\eta_n \right) + \eta 
&\le \Prb \left( \bbz \in \limsup_{n \to \infty} \mathcal{Z}^\eta_n\right) + \eta,\label{line:second_inequality:probability:2}
\end{align}
where the first inequality is the definition of $\mathcal{Z}^\eta_n$ and the second inequality follows from Fatou's lemma. Finally, we observe that $
\sum_{n = 1}^\infty \int_{\mathcal{Z}_n^\eta} d \bz \le \sum_{n = 1}^\infty \frac{1}{n^2} < \infty$, and so it follows from the Borel-Cantelli lemma that the equality $
\int_{\limsup \limits_{n \to \infty} \mathcal{Z}_n^\eta} d \bz = 0$ holds. 
Combining this equality with lines~\eqref{line:second_inequality:probability:1}-\eqref{line:second_inequality:probability:2} and the fact that $\bbz$ is absolutely continuous with respect to the Lebesgue measure on $\R^{2D}$ (Assumption~\ref{ass:continuous}), we have shown that $\lim_{\rho \downarrow 0}\sup_{\mathcal{Z} \in \mathscr{B}(\R^{2D}): \int_{\mathcal{Z}} d \bz \le \rho} \Prb(\bbz \in \mathcal{Z}) \le \eta. $ Since $\eta > 0$ was chosen arbitrarily, our proof of Claim~\ref{claim:second_inequality:probability:intermediary} is complete.\looseness=-1 
\hfill \halmos 
\end{proof}
 
Equipped with Claim~\ref{claim:second_inequality:probability:intermediary}, we are ready to prove Lemma~\ref{lem:proportion_bucket}. We observe that
\begin{align*}
\lim \limits_{S \to \infty} \sup \limits_{\mathcal{A}_S \subseteq \mathcal{B}_S: \frac{| \mathcal{A}_S|}{| \mathcal{B}_S|} \le \alpha_S } \sum_{(\blambda,\bmu) \in \mathcal{A}_S} p_S(\blambda,\bmu) &=\lim_{S \to \infty} \sup_{\mathcal{A}_S \subseteq \mathcal{B}_S: | \mathcal{A}_S| \le \alpha_S S^{2D}} \sum_{(\blambda,\bmu) \in \mathcal{A}_S} p_S(\blambda,\bmu) \\
&\le \lim_{S \to \infty} \sup_{\mathcal{Z} \in \mathscr{B}(\R^{2D}): \int_{\mathcal{Z}} d \bz \le \alpha_S } \Prb(\bbz \in \mathcal{Z})= 0.
\end{align*}
Indeed, the first equality follows from algebra and from the fact that $|\mathcal{B}_S| = S^{2D}$. The inequality follows from algebra and from the fact that $\int_{\{(\bell,\bu) \in \mathcal{F}: \blambda = \bsigma_S(\bell) \text{ and }\bmu = \bsigma_S(\bu) \}} d \bz = \frac{1}{S^{2D}}$ for all $(\blambda,\bmu)\in \mathcal{B}_S$. The second equality follows from Claim~\ref{claim:second_inequality:probability:intermediary} and from the fact that $\lim_{S \to \infty} \alpha_S = 0$.
\hfill \halmos 
\end{proof}

\begin{proof}{Proof of Lemma~\ref{lem:num_anchors}.}
Recall that $\bdelta \neq \bzero$. Therefore, 
\begin{align*}
\left| \mathfrak{A}_{S,M}(\bdelta) \right|&= \left| \bigcup_{d \in \{1,\ldots,D\}: \delta_d = 1} \left \{ (\blambda,\bmu) \in \mathcal{B}^{\ge}_{S,M}: \lambda_d = 1 \right \} \cup \bigcup_{d \in \{1,\ldots,D\}: \delta_d = -1} \left \{ (\blambda,\bmu) \in \mathcal{B}^{\ge}_{S,M}: \mu_d = S \right \} \right| \\
&\le \sum_{d \in \{1,\ldots,D\}: \delta_d = 1} \left| \left \{ (\blambda,\bmu) \in \mathcal{B}^{\ge}_{S,M}: \lambda_d = 1 \right \} \right| + \sum_{d \in \{1,\ldots,D\}: \delta_d = -1} \left| \left \{ (\blambda,\bmu) \in \mathcal{B}^{\ge}_{S,M}: \mu_d = S \right \} \right| \\
&\le \sum_{d \in \{1,\ldots,D\}: \delta_d = 1} S^{2D - 1} + \sum_{d \in \{1,\ldots,D\}: \delta_d = -1} S^{2D - 1} \le D S^{2D - 1}.
\end{align*}
The first equality is the definition of $\mathfrak{A}_{S,M}(\bdelta)$. The first inequality follows from algebra. The second inequality  follows from the fact that $\mathcal{B}^{\ge}_{S,M} \subseteq \mathcal{B}_S$. The third inequality follows from algebra.\looseness=-1
\hfill \halmos 
\end{proof}

\begin{proof}{Proof of Lemma~\ref{lem:unique_line}.}
Our proof consists of showing for all $(\blambda,\bmu) \in \mathcal{B}^{\ge}_{S,M}$ that there exists an anchor $(\bar{\blambda},\bar{\bmu}) \in \mathfrak{A}_{S,M}(\bdelta)$ that satisfies $(\blambda,\bmu) \in \mathcal{L}_{S,M}(\bar{\blambda},\bar{\bmu},\bdelta)$. Indeed, 
consider any arbitrary $(\blambda,\bmu) \in \mathcal{B}^{\ge}_{S,M}$, and define $\iota \triangleq \min \left \{ \iota' \in \{0,1,\ldots,S\}: (\blambda - \iota' \bdelta, \bmu - \iota' \bdelta) \in \mathcal{B}^{\ge}_{S,M} \right \}$ and $
(\bar{\blambda},\bar{\bmu}) \triangleq (\blambda - \iota \bdelta, \bmu - \iota \bdelta)$. 
It follows immediately from these definitions that $(\bar{\blambda},\bar{\bmu}) \in \mathcal{B}^{\ge}_{S,M}$, that $(\bar{\blambda},\bar{\bmu})$ is a $\bdelta$-anchor, and that $(\blambda,\bmu)$ is contained in the line emanating from anchor $(\bar{\blambda},\bar{\bmu}) \in \mathfrak{A}_{S,M}(\bdelta)$. Since $(\blambda,\bmu) \in \mathcal{B}^{\ge}_{S,M}$ was chosen arbitrarily, our proof is complete. 
\hfill \halmos 
\end{proof}

\begin{proof}{Proof of Lemma~\ref{lem:not_too_close}.}
Let $(\blambda,\bmu),(\blambda',\bmu') \in \mathcal{L}_{S,M}(\bar{\blambda},\bar{\bmu},\bdelta)$, and suppose there exists an integer $\iota \in \{1,\ldots, M-1 \}$ such that 
 $(\blambda',\bmu') = (\blambda,\bmu) + \iota (\bdelta,\bdelta)$. We observe that
\begin{align*}
\textsc{Face}(\blambda,\bmu,\bdelta) 
&= \left \{ \bsigma: \; \begin{aligned}
\sigma_d &= \lambda'_d - \iota \delta_d && \text{ for all } d \in \{1,\ldots,D\} \text{ such that } \delta_d = -1 \\ 
\sigma_d &= \mu'_d - \iota \delta_d && \text{ for all } d \in \{1,\ldots,D\} \text{ such that } \delta_d = 1 \\ 
\sigma_d &\in \{\lambda'_d - \iota \delta_d + 1,\ldots,\mu_d' - \iota \delta_d- 1\} &&\text{ for all } d \in \{1,\ldots,D\} \text{ such that } \delta_d = 0 
\end{aligned}\right \} \\
&= \left \{ \bsigma: \; \begin{aligned}
\sigma_d &= \lambda'_d + \iota && \text{ for all } d \in \{1,\ldots,D\} \text{ such that } \delta_d = -1 \\ 
\sigma_d &= \mu'_d - \iota && \text{ for all } d \in \{1,\ldots,D\} \text{ such that } \delta_d = 1 \\ 
\sigma_d &\in \{\lambda'_d + 1,\ldots,\mu_d' - 1\} &&\text{ for all } d \in \{1,\ldots,D\} \text{ such that } \delta_d = 0 
\end{aligned}\right \} \\
&\subseteq \left \{ \bsigma: \; 
\sigma_d \in \{\lambda'_d + 1,\ldots,\mu_d' - 1\} \text{ for all } d \in \{1,\ldots,D\}
\right \} = \textsc{Face}(\blambda',\bmu',\bzero).
\end{align*}
The first equality follows from the definition of $\textsc{Face}(\blambda,\bmu,\bdelta)$ and from the fact that $(\blambda',\bmu') = (\blambda,\bmu) + \iota (\bdelta,\bdelta)$. The second equality follows from substituting $\delta_d$. To see why the inclusion holds, we observe that  $(\blambda',\bmu') \in \mathcal{B}^{\ge}_{S,M}$ implies that $ \mu_d' - \lambda_d' \ge M$ for all $d \in \{1,\ldots,D\}$. Therefore, for all $d \in \{1,\ldots,D\}$, it follows from the fact that $\mu'_d - \lambda'_d \ge M $ and from the fact that $\iota \in \{1,\ldots,M-1 \}$ that 
$\lambda_d' + \iota \in \left \{ \lambda_d' + 1,\ldots, \lambda_d' + M - 1 \right \} \subseteq \left \{ \lambda_d' + 1,\ldots, \mu_d' - 1 \right \}$ 
and $
\mu_d' - \iota \in \left \{ \mu_d' - (M-1),\ldots,\mu_d' - 1 \right \} \subseteq \left \{ \lambda_d' + 1,\ldots, \mu_d' - 1 \right \}. $ 
The third equality is the definition of $\textsc{Face}(\blambda',\bmu',\bdelta)$.\looseness=-1 
\hfill \halmos
\end{proof}

\bibliographystyle{ormsv080}
\bibliography{references}

\ECSwitch
\renewcommand{\thesection}{Appendix}

\vspace{-8mm}
\begin{APPENDICES}
\section{Necessity of Assumption~\ref{ass:continuous}} \label{sec:discussion:continuous}
We show in this section that Assumption~\ref{ass:continuous} is necessary for asymptotic consistency, in the sense that relaxing the absolute continuity requirement can invalidate Theorem~\ref{thm:main}.
\begin{example} \label{example:1}
Let $D = 2$ and $\epsilon \in [0,\sfrac{1}{4}]$. Let $\bbx \in [0,1]^2$ be a random vector drawn uniformly from the perimeter of the circle centered at $(\sfrac{1}{2},\sfrac{1}{2}) $ with radius $\sfrac{1}{4}$. Let $\bbl = \bbx - \mathbf{1}
\epsilon$ and $\bbu = \bbx + \mathbf{1} \epsilon$. If $\epsilon > 0$, then we observe that  $\Prb(L_d < U_d) = 1$ for all $d \in \{1,2\}$. However, we observe for all $\epsilon \in [0,\sfrac{1}{4}]$ that $(\bbl,\bbu)$ is not drawn from an probability distribution that is absolutely continuous with respect to the Lebesgue measure on $\R^{2D}$; indeed,  Figure~\ref{fig:example_1}a shows that $\Prb((\bbl,\bbu) \in \mathcal{Z}_L \times \mathcal{Z}_U)  = 1$ for $\mathcal{Z}_L \triangleq \{(\sfrac{1}{4} \cos(\theta) + \sfrac{1}{2} - \epsilon,\sfrac{1}{4} \sin(\theta)+ \sfrac{1}{2} - \epsilon): \theta \in [0,2\pi) \}$ and $\mathcal{Z}_U \triangleq \{(\sfrac{1}{4} \cos(\theta) + \sfrac{1}{2} + \epsilon,\sfrac{1}{4} \sin(\theta)+ \sfrac{1}{2} + \epsilon): \theta \in [0,2\pi) \}$, and we observe from Figure~\ref{fig:example_1}a that the Lebesgue measure with respect to $\R^{2D}$ of $\mathcal{Z}_L \times \mathcal{Z}_U \subset \R^{2D}$ is equal to zero.  Hence, we observe for every $\epsilon \in [0,\sfrac{1}{4}]$ that 
Assumption~\ref{ass:continuous} is violated.\looseness=-1 

To show that the asymptotic consistency guarantee from Theorem~\ref{thm:main} is invalidated for this example, let the set of feasible prices be $\mathcal{P} = \{\sfrac{1}{3},\sfrac{1}{2}\}$, and let $V$ be a random variable that is independent of $(\bbl,\bbu)$ and satisfies $\Prb(V = \sfrac{1}{3}) = \Prb(V = \sfrac{1}{2}) = \sfrac{1}{2}$. Since $V$ is independent of $(\bbl,\bbu)$, it is easy to see that there exists an optimal solution for \eqref{prob:opt} that satisfies $\pi^*(\bx) = \sfrac{1}{3}$ for all $\bx \in [0,1]^2$ and that the objective value of \eqref{prob:opt} is $\nu^* = \sfrac{1}{3}$.\footnote{The fact that $V$ is independent of $(\bbl,\bbu)$ implies that either $\pi_1(\cdot) = \frac{1}{3}$ or $\pi_2(\cdot) = \frac{1}{2}$ is an optimal solution for \eqref{prob:opt}. We observe that $J^*(\pi_1) = \frac{1}{2} \Exp \left[ R^{\pi_1}(\bbl,\bbu,\frac{1}{3}) \right] + \frac{1}{2} \Exp \left[ R^{\pi_1}(\bbl,\bbu,\frac{1}{2}) \right] = \frac{1}{2} \cdot \frac{1}{3} + \frac{1}{2} \cdot \frac{1}{3} = \frac{1}{3}$ and $J^*(\pi_2) = \frac{1}{2} \Exp \left[ R^{\pi_2}(\bbl,\bbu,\frac{1}{3}) \right] + \frac{1}{2} \Exp \left[ R^{\pi_2}(\bbl,\bbu,\frac{1}{2}]) \right] = \frac{1}{2} \cdot  0 + \frac{1}{2} \cdot \frac{1}{2} = \frac{1}{4}$. Thus, we conclude that an optimal solution for \eqref{prob:opt} is $\pi^* = \pi_1$ and the objective value of \eqref{prob:opt} is $\nu^* = \frac{1}{3}$. } However, Figure~\ref{fig:example_1}b illustrates that there almost surely exists a pricing policy $\pi: [0,1]^2 \to \{\sfrac{1}{3},\sfrac{1}{2}\}$ that satisfies the equality $R^\pi(\bbl^i,\bbu^i,\sfrac{1}{3}) = \sfrac{1}{3}$ for all $i \in \{1,\ldots,N\}$ such that $V^i = \sfrac{1}{3}$ and satisfies the equality $R^\pi(\bbl^i,\bbu^i,\sfrac{1}{2}) = \sfrac{1}{2}$ for all $i \in \{1,\ldots,N\}$  such that $V^i = \sfrac{1}{2}$. Hence, we conclude from the strong law of large numbers that $\lim_{N \to \infty} \widehat{\nu}_N = (\sfrac{1}{2})(\sfrac{1}{3}) + (\sfrac{1}{2}) (\sfrac{1}{2}) = \sfrac{5}{12} \neq \nu^*$ almost surely. 
\end{example}
\begin{figure}[t]
\centering
\FIGURE{
\begin{minipage}{\linewidth}
\centering
\subfloat[]{
\includegraphics[width=0.47\textwidth]{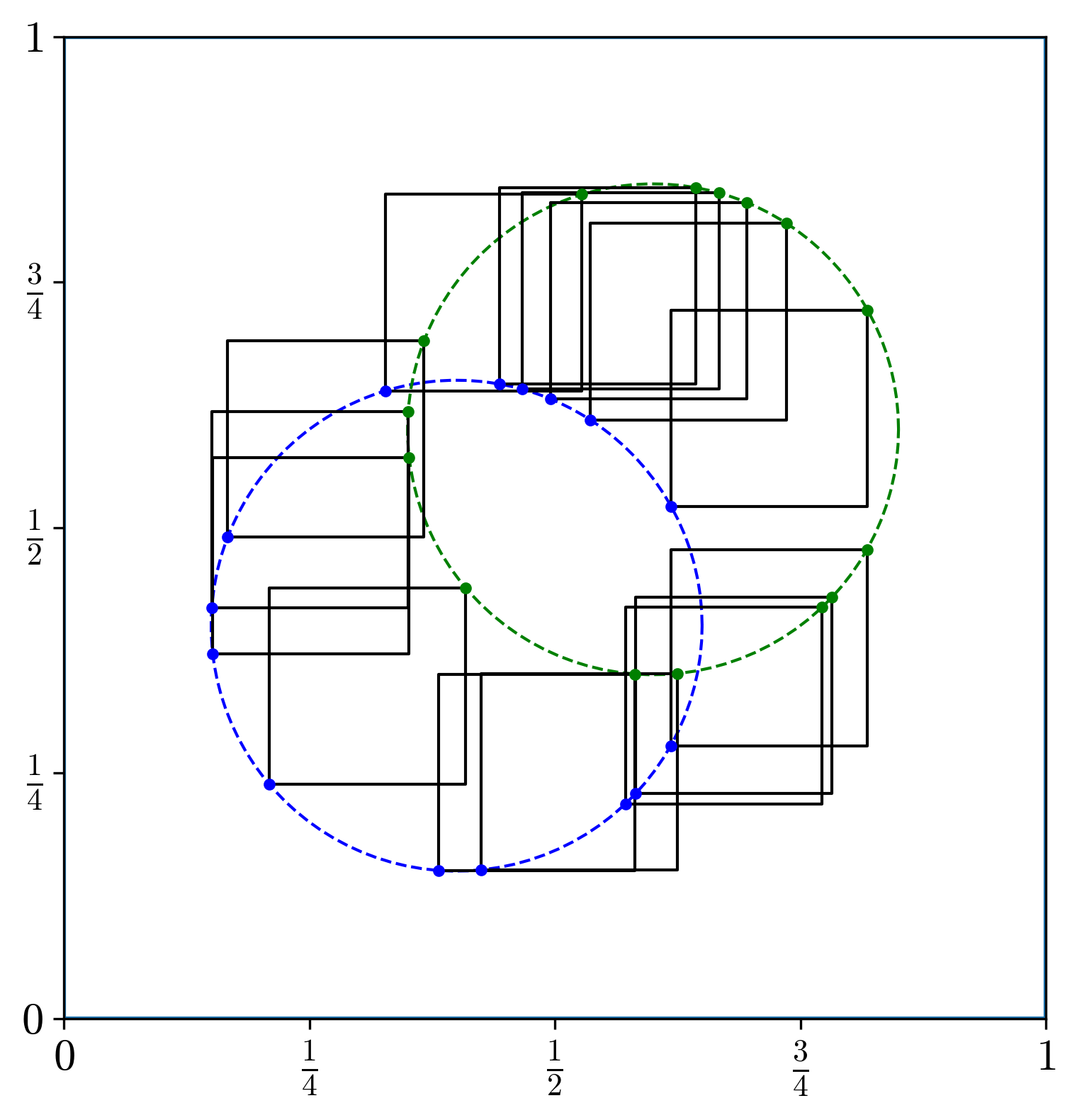}
}
\subfloat[
]{
\includegraphics[width=0.47\textwidth]{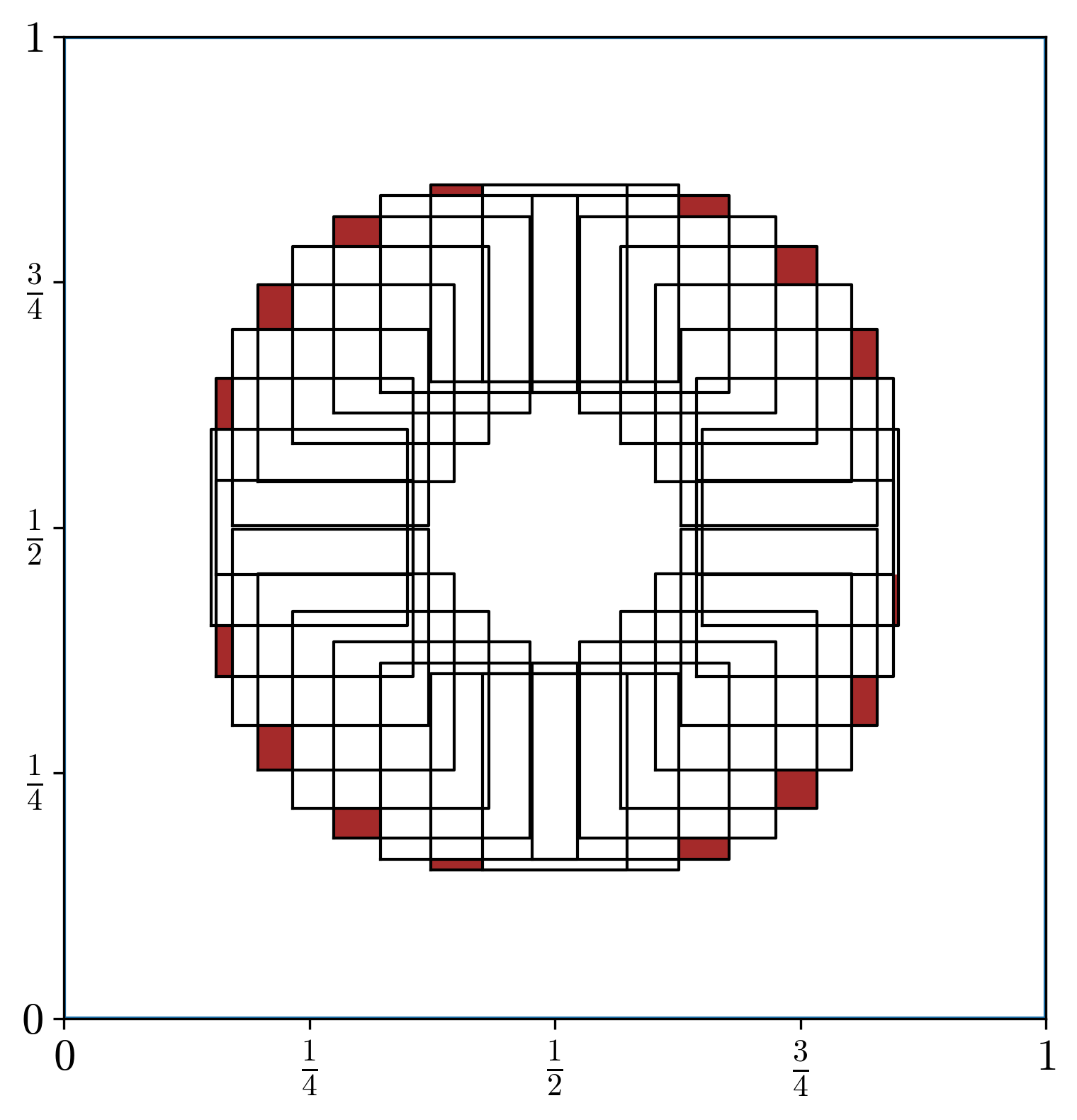}
}
\vspace{0.5em}
\end{minipage}
}
{Visualization of Example~\ref{example:1}\label{fig:example_1}} 
{Both figures show case where $\epsilon = 0.1$. \emph{(a)} $N = 15$. Blue dots show $\bbl^1,\ldots,\bbl^{15}$ and green dots show $\bbu^1,\ldots,\bbu^{15}$. Blue and green dashed lines show the sets $\mathcal{Z}_L$ and $\mathcal{Z}_U$, respectively. \emph{(b)} $N = 30$. The figure shows a pricing policy $\pi: [0,1]^2 \to \{\sfrac{1}{3},\sfrac{1}{2}\}$ in which the low price $\pi(\bx) = \sfrac{1}{3}$ is offered if and only if $\bx \in [0,1]^2$ is in a red shaded region.  We observe that this pricing policy offers the low price only in the outermost corners of alternating squares, and that the low price is not offered in any of the other squares. Following this strategy, we conclude for a given sample $(\bbl^1,\bbu^1,V^1),\ldots,(\bbl^N,\bbu^N,V^N)$ that we construct a pricing policy that satisfies the equality $\min_{\bx \in [\bbl^i,\bbu^i]} \pi(\bx) = \sfrac{1}{3}$ for all $i \in \{1,\ldots,N\}$  such that $V^i = \sfrac{1}{3}$ and satisfies the equality $\min_{\bx \in [\bbl^i,\bbu^i]} \pi(\bx) = \sfrac{1}{2}$ for all $i \in \{1,\ldots,N\}$   such that $V^i = \sfrac{1}{2}$. Note for this figure that $\bbx^1,\ldots,\bbx^{30}$ are equally spaced around the perimeter of the circle, rather than drawn uniformly over the perimeter, for the sake of visual clarity. }
\end{figure}

\section{Experiment Details} \label{sec:MILP}
Our numerical experiment in \S\ref{sec:numerics} involves solving \eqref{prob:saa}. In this appendix, we provide the mixed-integer linear program reformulation of~\eqref{prob:saa} under Assumption~\ref{ass:discrete} that is used in the numerical experiment.\looseness=-1

First, let $\mathcal{U}^i = [\bbl^i, \bbu^i]$ be the hyperrectangle of buyer $i$'s features. We then define the following collection of subsets of buyers: 
\begin{align*}
\mathfrak{I} \triangleq \bigg\{\mathcal{I} \in 2^{\{1,\ldots,N\}}: \bigcap_{i \in\mathcal{I}} \mathcal{U}^i  \setminus \bigcup_{i \in \{1,\ldots,N\} \setminus\mathcal{I}} \mathcal{U}^i \neq \emptyset \bigg\}.
\end{align*}
For each $\mathcal{I} \in \mathfrak{I}$, let $\mathcal{W}_{\mathcal{I}} \triangleq \bigcap_{i \in\mathcal{I}} \mathcal{U}^i  \setminus \bigcup_{i \in \{1,\ldots,N\} \setminus\mathcal{I}} \mathcal{U}^i$ denote the subset of features corresponding to $\mathcal{I}$. Note that $\{\mathcal{W}_{\mathcal{I}}:\mathcal{I} \in \mathfrak{I}\}$ is a collection of disjoint sets whose union is equal to $\cup_{i \in \{1,\ldots,N\}} \mathcal{U}^i$. 

Next, for each buyer $i$, let $\mathfrak{I}^{i}$ denote the subset of $\mathfrak{I}$ such that $\mathcal{U}^i = \cup_{\mathcal{I} \in \mathfrak{I}^{i}} \mathcal{W}_{\mathcal{I}}$. We observe that 
\begin{align*}
\sup \limits_{\pi: \mathcal{X} \rightarrow \mathcal{P}} \frac{1}{N} \sum \limits_{i=1}^N \mathscr{R}^\pi\left(\bbl^i,\bbu^i, V^i \right)&=\sup \limits_{\pi: \mathcal{X} \rightarrow \mathcal{P}} \frac{1}{N} \sum \limits_{i=1}^N \mathbb{I} \left \{\inf_{\bx \in [\bbl^i,\bbu^i]} \pi (\bx ) \le V^i \right \} \cdot \inf_{\bx \in [\bbl^i,\bbu^i]}\pi (\bx ) \\
&=\sup \limits_{\pi: \mathcal{X} \rightarrow \mathcal{P}} \frac{1}{N} \sum \limits_{i=1}^N \mathbb{I} \left \{\min_{\mathcal{I} \in \mathfrak{I}^i} \inf_{\bx \in \mathcal{W}_{\mathcal{I}}} \pi (\bx ) \le V^i \right \} \cdot \min_{\mathcal{I} \in \mathfrak{I}^i} \inf_{\bx \in \mathcal{W}_{\mathcal{I}}}\pi (\bx ) \\
&=\sup \limits_{\tilde{\pi}: \mathfrak{I} \rightarrow \mathcal{P}} \frac{1}{N} \sum \limits_{i=1}^N \mathbb{I} \left \{\min_{\mathcal{I} \in \mathfrak{I}^i} \tilde{\pi}(\mathcal{I})\le V^i \right \} \cdot \min_{\mathcal{I} \in \mathfrak{I}^i} \tilde{\pi}(\mathcal{I}),
\end{align*}
where the first equality follows from the definition of revenue function, the second equality is due to $\mathcal{U}^i:= [\bbl^i,\bbu^i]= \cup_{\mathcal{I} \in \mathfrak{I}^{i}} \mathcal{W}_{\mathcal{I}}$, and the last equality is because for any pricing policy $\pi(\cdot)$, we can define an equivalent policy based on set $\mathfrak{I}$ as $\tilde{\pi}(\mathcal{I}):=\inf_{\bx \in \mathcal{W}_{\mathcal{I}}}\pi (\bx)$.
Hence, we have the following mixed-integer linear program reformulation:
\[
\begin{array}{c@{\quad}l@{\quad}l@{\quad}}
\underset{\bm{b}, \blambda}{\textnormal{maximize}} & \displaystyle \dfrac{1}{N} \sum_{i = 1}^N \sum_{k = 1}^K p_k b_{ik} \\[5mm]
\textnormal{subject to} 
& b_{ik} \leq \mathbb{I}\{p_k\le V^i\} &~\forall i = 1,\dots,N, \; k = 1,\dots,K \\[5mm]
& \displaystyle b_{ik} \leq \sum_{\kappa = 1,\dots,K: \kappa \geq k} \lambda_{\mathcal{I}\kappa} &~\forall i = 1,\dots,N, \; k = 1,\dots,K, \;\mathcal{I} \in \mathfrak{I}: i \in \mathfrak{N}_{\mathcal{I}} \\[5mm]
& \displaystyle b_{ik} \leq \sum_{\mathcal{I} \in \mathfrak{I}: i \in \mathfrak{N}_{\mathcal{I}}} \lambda_{\mathcal{I}k} &~\forall i = 1,\dots,N, \; k = 1,\dots,K \\[5mm]
& \displaystyle \sum_{k = 1}^K \lambda_{\mathcal{I}k} = 1 &~\forall\mathcal{I} \in \mathfrak{I} \\[5mm]
& \displaystyle \sum_{k = 1}^K b_{ik} \leq 1 &~\forall i = 1,\dots,N \\[5mm]
& \bm{b} \in \{0,1\}^{N \times K}, \; \blambda \in \{0,1\}^{|\mathfrak{I}| \times K}.
\end{array}
\]
Here, for each $\mathcal{I} \in \mathfrak{I}$, $\mathfrak{N}_{\mathcal{I}} = \{i \in \{1,\dots,N\}: \mathcal{W}_\mathcal{I} \subseteq \mathcal{U}^i\}$ denotes the largest subset of buyers whose hyperrectangle of features contains $\mathcal{W}_{\mathcal{I}}$, $\lambda_{\mathcal{I}k} = 1$ if and only if a region is offering the price $p_k$, while $b_{ik} = 1$ if and only if
\[
k \in \argmax_{\kappa = 1,\dots,K}\left\{\kappa: \min_{\mathcal{I} \in \mathfrak{I}: i \in \mathfrak{N}_{\mathcal{I}}} p_\kappa \lambda_{\mathcal{I}k} \le V^i\right\} = \argmax_{\kappa = 1,\dots,K}\left\{\kappa: p_\kappa \cdot \min_{\mathcal{I} \in \mathfrak{I}: i \in \mathfrak{N}_{\mathcal{I}}} \lambda_{\mathcal{I}k} \le V^i\right\}.
\]
\end{APPENDICES}
\end{document}